\documentclass[11pt, a4paper]{article}
\usepackage{amsmath,amssymb,amsfonts}
\usepackage{amsthm}
\usepackage{verbatim}
\usepackage{a4wide}
\usepackage{epsfig,epic,eepic,graphicx}
\usepackage{graphicx}
\usepackage[active]{srcltx}
\usepackage{enumerate}

\newtheorem{theorem}{Theorem}
\newtheorem{lemma}{Lemma}
\newtheorem{proposition}{Proposition}

\newtheorem{remark}{Remark}

\makeatletter

\@addtoreset{equation}{section}
\makeatother

\newcommand{\T}{{\mathbb T}}
\newcommand{\N}{{\mathbb N}}
\newcommand{\Z}{{\mathbb Z}}

\newcommand{\R}{{\mathbb R}}
\newcommand{\C}{{\mathbb C}}
\newcommand{\pa}{{\partial}}
\newcommand{\na}{{\nabla}}

\newcommand{\eps}{{\varepsilon}}
\newcommand{\be}{\begin{equation}}
\newcommand{\ee}{\end{equation}}

\newcommand{\Om}{\Omega}

\newcommand{\dv}{\mathrm{div}\;}
\newcommand{\om}{\omega}

\newcommand{\tv}{\tilde v}

\newcommand{\Supp}{\mathrm{Supp}\;}
\newcommand{\ba}{\begin{aligned}}
\newcommand{\ea}{\end{aligned}}
\newcommand{\sgn}{\mathrm{sgn}}
\newcommand{\cH}{\mathcal H}
\newcommand{\cF}{\mathcal F}
\newcommand{\cV}{\mathcal V}

\def\div{\hbox{div  }}

\def\Re{{\rm Re}  \,}

\def\inside{\quad \mbox{in} \quad}
\def\defin{\: := \: }

\title{Nonlinear boundary layers for rotating fluids} 
\author{\footnote{Sorbonne Universit\'es, UPMC Univ Paris 06, CNRS, UMR 7598, Laboratoire Jacques-Louis Lions, 4, place Jussieu 75005, Paris, France.} Anne-Laure Dalibard 
and 
 \footnote{Universit\'e Paris Diderot, Sorbonne Paris Cit\'e, Institut de Math\'ematiques de Jussieu-Paris Rive Gauche, UMR 7586, F- 75205 Paris, France} David G\'erard-Varet }
\date{}

\begin{document}
\bibliographystyle{amsplain}
\maketitle

\begin{abstract}
We investigate the behavior of rotating incompressible flows near a non-flat horizontal bottom.  In the flat case, the velocity profile is given explicitly by a simple linear ODE. When bottom variations are taken into account, it is governed by a nonlinear PDE system, with far less obvious mathematical properties. We establish the well-posedness of this system and the asymptotic behavior of the solution away from the boundary. In the course of the proof, we investigate in particular the action of pseudo-differential operators in non-localized Sobolev spaces. Our results extend the older paper \cite{Gerard-Varet:2003b}, restricted to periodic variations of the bottom. It ponders on the recent linear analysis carried in \cite{DalibardPrange}. 
\end{abstract}

 
\section{Introduction}
The general concern of this paper is the effect of rough walls on fluid flows, in a context where the rough wall has very little structure. This effect is important in several problems, like transition to turbulence or drag computation.  For instance, understanding the connection between roughness and drag is crucial for microfluidics, because friction at solid boundaries is a major factor of energy loss in microchannels.  This issue has been much studied over recent years, through both theory and experiments \cite{Lauga,Bocquet_Barrat}. Conclusions are ambivalent. On the one hand, rough surfaces may increase the friction area, and thus enstrophy dissipation. On the other hand, recent experiments have shown that rough hydrophobic surfaces may lead to drag decrease:  air bubbles can be trapped  in the humps of the roughness, generating some slip \cite{Vino,Ybert}. 

\medskip
Mathematically, these problems are often tackled by a homogenization approach. Typically, one considers Stokes equations over a rough plate, modeled by an oscillating boundary of small wavelength and amplitude: 
\begin{equation} \label{Gammaeps}
 \Gamma^\eps :  \: x_3 = \eps \gamma(x_1/\eps, x_2/\eps), \quad \eps \ll 1, 
 \end{equation}  
where the function $\gamma = \gamma(y_1,y_2)$ describes the roughness pattern. Within this formalism, the understanding of roughness-induced effects comes down to an asymptotic problem, as $\eps \rightarrow 0$. The point is to derive effective boundary conditions at the flat plate $\Gamma^0$, retaining in this boundary condition an averaged effect of the roughness. We refer to the works \cite{Achdou:1995,Achdou:1998,Achdou:1998a,AmBrLe,Jager:2001,Jager:2003,Neuss,BrMi,MikNec} on this topic. 
In all of these works, a restrictive hypothesis is made, namely periodicity of the roughness pattern $\gamma$. This hypothesis simplifies greatly the construction of the so-called boundary layer corrector, describing the small-scale variations of the flow near the boundary. This corrector is an analogue of the cell corrector in classical homogenization of heterogeneous media. 

\medskip
The main point and difficulty is the mathematical study of the boundary layer equations, that are satisfied formally by the boundary layer corrector. 
When $\gamma$ is periodic in $y_1,y_2$, the solution of the boundary layer system is itself sought periodic, so that well-posedness and qualitative properties of the system are easy to determine. When the periodicity structure is relaxed, and replaced by general ergodicity properties, the analysis is still possible, but much more involved, as shown in \cite{BaGe,DGV_CMP,DGVNMnoslip}. A key feature of these articles is the linearity of the boundary layer system: after the rescaling $y = x / \eps$, it is governed by a Stokes equations in the boundary layer domain
\begin{equation} \label{defOmegabl}
\Omega_{bl} = \{ y, \: y_3 > \gamma(y_1,y_2)\}. 
\end{equation}
It thus reads 
\begin{equation} \label{Stokes}
\left\{
\begin{aligned}
-\Delta v + \na p & = 0, \quad  \mbox{in} \:  \Omega_{bl}, \\
\div v & = 0,  \quad  \mbox{in} \:  \Omega_{bl}, \\
v\vert_{\pa \Omega_{bl}} & = \phi
\end{aligned}
\right.
\end{equation}
for some Dirichlet boundary data $\phi$ that has no decay as $y_1,y_2$ go to infinity, but no periodic structure. As a consequence, spaces of infinite energy, such as $H^{s}_{uloc}$, form a natural functional setting for such equations.

\medskip
A natural challenge is to extend this type of analysis to nonlinear systems. This is the goal of the present paper.   
Namely, we will study a nonlinear boundary layer system that describes a  rotating fluid near a rough boundary.  The dynamics of rotating fluid layers is notably relevant in the context of geophysical flows, for which the Earth's rotation plays a dominant role. The system under consideration reads 
\begin{equation}  \label{BL}
\left\{
\begin{aligned}
v \cdot \na v  + \na p  + e \times v - \Delta v  & =  0 \quad \mbox{in} \: \Omega_{bl} \\
\div v & = 0  \quad \mbox{in} \: \Omega_{bl} \\
v\vert_{\pa \Omega_{bl}} & = \phi.
\end{aligned}
\right.
\end{equation}
These are the incompressible Navier-Stokes equations written in a rotating frame, which is the reason for the extra Coriolis force $e \times u$, where $e = e_3 = (0,0,1)^t$.  The equations in \eqref{BL} can be obtained through an asymptotics of the full rotating fluid system 
\begin{equation} \label{NSC}
 \textrm{Ro} (\pa_t u + u \cdot \na u)  + e \times u -  E \Delta u = 0, \quad \div u = 0
 \end{equation}
where $\textrm{Ro}$ and $\textrm{E}$ are the so-called Rossby and Ekman numbers. These parameters are small in many applications. In the vicinity of the rough boundary \eqref{Gammaeps}, and in the special case where 
\begin{equation} \label{scaling}
 E  \sim  \eps^2, \quad \textrm{Ro}  \sim \eps 
\end{equation}
it is natural to look for an asymptotic behavior of the type  
$$ u^\eps(t,x) \sim v(t,x_1,x_2, x/\eps) $$ 
where $v = v(t,x_1,x_2,y)$, $y \in \Omega_{bl}$. Injecting this Ansatz in  \eqref{NSC} yields (\ref{BL}a)-(\ref{BL}b), where the ``slow variables" $(t,x_1,x_2)$ are only parameters and are eluded. 

\medskip
The main goal of this paper is to construct a solution $v$ of system \eqref{BL}, under no structural assumption on $\gamma$. We shall moreover provide information on the behavior of $v$ away from the boundary. We will in this way generalize article  \cite{Gerard-Varet:2003b}  by the second author  in which periodic roughness was considered. See also \cite{DGV_Dormy}.  Before stating the main difficulties and results of our study, several remarks are in order: 
\begin{enumerate}
\item
The choice of the scaling \eqref{scaling}, that leads to the derivation of the boundary layer system,  may seem  peculiar. It is however the richest possible, as it retains all terms in the equation for the boundary layer. All other scaling would provide a degeneracy of system \eqref{BL}. 

\item
In the flat case, that is for the roughness profile $\gamma = 0$, and for $\phi = (\phi_1,\phi_2,0)$, with $\phi_1, \phi_2$ independent of $y$, the solution of \eqref{NSC} is explicitly given in complex form by 
\begin{equation} 
(v_1 + i v_2)(y) = (\phi_1 + i \phi_2) \exp\left(-(1+i)y_3/\sqrt{2} \right), \quad v_3 = 0. 
\end{equation}
This profile, sometimes called Ekman spiral, solves the linear ODE 
$$ e \times v  -  \pa^2_{3} v = 0. $$
Considering roughness turns this linear ODE into a nonlinear PDE, and as we will see,  changes drastically the properties of the solution. 
\item
Rather than the  Dirichlet condition $v\vert_{\pa \Omega_{bl}} = \phi$, some slightly different settings could be considered: 
\begin{itemize}
\item One could for instance prescribe a homogenous Dirichlet condition  $v\vert_{\pa \Omega_{bl}} = 0$, and add a source term with enough decay in $y_3$. This would correspond to a localized forcing of the boundary layer. 
\item One could replace the Dirichet condition by a Navier condition, that is of the type: 
$$D(u)n \times n \vert_{\pa \Omega_{bl}} = f, \quad u \cdot n\vert_{\pa \Omega_{bl}} = 0,$$
with $D(u)$ the symmetric part of $\na u$, and $n$ the normal unit vector at the boundary. For instance, one could  think of \eqref{Gammaeps} as modeling  an oscillating free surface, under the rigid lid approximation. In this context, the Navier condition would model a wind forcing, and the boundary layer domain would model the water below the free surface (changing the direction of the vertical axis). We refer to \cite{Ped} for some similar modeling, and to \cite{Casado,Bucur_Feireisl,Bonnivard,DaGe} for the treatment of such Navier condition. As shown in those papers, some hypothesis on the non-degeneracy of the roughness is necessary to the mathematical analysis. 
\end{itemize}
However, our analysis does not  extend to the important case of a inhomogeneous Dirichlet condition at infinity, which models a boundary layer driven by an external flow. For linear systems, one can in general lift this Dirichlet data at infinity, and recover the case of a Dirichlet data at the bottom boundary, like in \eqref{Stokes}. But for our nonlinear system \eqref{BL}, this lift would lead to the introduction of an additional drift term in the momentum equation, which would break down its rotational invariance. 
\end{enumerate}

\section{Statement of the results}
Our main result is a well-posedness theorem for the boundary layer system  \eqref{BL}, where $\phi$ is a  given boundary data, with no decay tangentially to the boundary, and  satisfying $\phi \cdot n\vert_{\pa \Omega_{bl}} = 0$.  As usual in the theory of steady Navier-Stokes equations, the well-posedness will be obtained under a smallness hypothesis. We first introduce, for  any unbounded $\Omega \subset \R^d$,   the spaces 
$$ L^2_{uloc}(\Omega) = \{ f, \quad \sup_{k \in \Z^d} \int_{B(k, 1) \cap \Omega} |f|^2 < +\infty\}, $$
$$ \: \mbox{and for all}  \: m \ge 0, \quad  \: H^m_{uloc}(\Omega) = \{ f, \pa^\alpha f \in L^2_{uloc}(\Omega), \quad \forall \alpha \le m \}. $$
These spaces are of course Banach spaces when endowed with their natural norms. 
\begin{theorem} \label{Thm_WP}
Let  $\gamma$ be bounded and Lipschitz, $\Omega_{bl}$ defined in \eqref{defOmegabl}. There exists $\delta_0, C > 0$, such that: for all  $\phi \in H^2_{uloc}(\pa \Omega_{bl})$, satisfying $\phi \cdot n\vert_{\pa \Omega_{bl}} = 0$ and $\| \phi \|_{H^2_{uloc}} \le \delta_0$,  system has a unique solution 
$(v,p)$ with 
$$ (1+y_3)^{1/3} v  \in H^1_{uloc}(\Omega_{bl}), \quad (1+y_3)^{1/3} p \in  L^2_{uloc}(\Omega_{bl}), $$
and  
$$ \| (1+y_3)^{1/3} v \|_{H^1_{uloc}} + \| (1+y_3)^{1/3} p \|_{L^2_{uloc}} \le C \|\phi\|_{H^2_{uloc}}. $$
\end{theorem}
This theorem generalizes the result of \cite{Gerard-Varet:2003b}, dedicated to the case of periodic roughness pattern $\gamma$. In this case, the analysis is much easier, as the solution $v$ of \eqref{BL} is itself periodic in $y_1,y_2$.  Through standard arguments, one can then build a solution $v$ satisfying
$$ \int_{\T^2} \int_{y_3 > \gamma(y_1,y_2)} | \na v |^2   < +\infty $$
Moreover, one can establish {\em Saint-Venant estimates} on $v$, namely exponential decay estimates for $v$ as $y_3$ goes to infinity. This exponential decay is related to the periodicity in the horizontal variables, which provides a Poincar\'e inequality for functions with zero mean in $x_1$. When the periodicity assumption is removed, one expects the exponential convergence to be no longer true : this has been notably discussed in article \cite{DGVNMnoslip,Prange}, in the context of the Laplace or the Stokes equation near a rough wall. It is worth noting that in such context, the convergence can be arbitrarily slow. In fact, there is in general no convergence  when no ergodicity assumption on $\gamma$ is made. A remarkable feature of our theorem for rotating flows is that decay to zero persists, despite the nonlinearity, and without any ergodicity assumption on $\gamma$. We emphasize that this decay comes from the rotation term. However, exponential decay is replaced by polynomial decay, with rate $O(y_3^{-1/3})$ for $v$. 

\medskip
Let us comment on the difficulties associated with Theorem \ref{Thm_WP}. Of course, the first issue is that  the data $\phi$ does not decay as $(y_1,y_2)$ goes to infinity, so that the solution $v$ is not expected to decay in the horizontal directions. 
If $\Omega_{bl}$ were replaced by 
$$\Omega_{bl}^M \defin \{y,  M > y_3 > \gamma(y_1,y_2)\}, \quad M > 0$$ 
 together with a Dirichlet condition at the upper boundary, one could build a solution $v$ in $H^1_{uloc}(\Omega_{bl}^M)$, adapting ideas of Ladyzenskaya and  Solonnikov  \cite{LS} on  Navier-Stokes flows in tubes. Among those ideas,  the main one is to obtain an {\it a priori} differential inequality on the  local energy 
$$ E(t)  \defin \int_{\{|(y_1,y_2)| \le t\}}  \int_{\{M > y_3 > \gamma(y_1,y_2)\}} |\na v|^2.  $$
Namely, one shows an inequality of the type
$$ E(t) \le C_M ( E'(t) + E'(t)^{3/2} + t^2). $$
However, the derivation of this differential inequality relies on the Poincar\'e inequality between two planes, or in other words on the fact that $\Omega_{bl}^M$ has a bounded direction. For the boundary layer domain $\Omega_{bl}$, this is no longer true, and  no  a priori bound can be obtained in this way. Moreover, contrarily to what happens for the Laplace equation, one can not rely on maximum principles to get an $L^\infty$ bound.   

\medskip
Under a periodicity assumption on $\gamma$, one can restrict the domain to the periodic slab: $ \{ y, \: (y_1,y_2) \in \T^2, \quad y_3 > \gamma(y_1,y_2)\}$. In this manner, one has again a domain with a bounded direction (horizontal rather than vertical). The analogue of Solonnikov's idea in this context leads to the Saint Venant estimates mentioned above. This allows to prove well-posedness of the boundary layer system. However, this approach does not work in our framework, where no structure is assumed on the roughness profile $\gamma$.

\medskip
 For the Stokes boundary layer flow: 
\begin{equation} \label{StokesBL} 
-\Delta v + \na p = 0, \quad \div v = 0 \quad \inside \Omega_{bl}, \quad v\vert_{\pa \Omega_{bl}} = v_0 
\end{equation}
this problem is overcome in paper \cite{DGVNMnoslip}, by N. Masmoudi and the second author. The main idea there is to get back to the domain $\Omega_{bl}^M$, by imposing a so-called transparent boundary condition at $y_3 = M$. This transparency condition involves the Stokes analogue of the Dirichlet-to-Neumann operator, and despite its non-local nature (contrary to the Dirichlet condition) allows then to apply  the method of Solonnikov. We refer to \cite{DGVNMnoslip} for more details\footnote{Actually, article \cite{DGVNMnoslip}  is concerned with the 2d case. For adaptation to 3d, we refer to \cite{DalibardPrange}.}.  Of course, the use of an explicit transparent boundary condition at $y_3 = M$ is possible because $v$ satisfies a homogeneous Stokes equation in the half-space $\{y_3 > M\}$, which gives access to explicit formulas.

\medskip
Such simplification does not occur in the context of our rotating flow system: in particular, the main issue is the quasilinear term $u \cdot \na u$ in system \eqref{BL}, in contrast with previous linear studies. In fact, even without this convective term, the analysis is uneasy. In other words, the Coriolis-Stokes problem \eqref{BL-lin}
\begin{equation}  \label{BL-lin}
\left\{
\begin{aligned}
 e \times v  + \na p - \Delta v  & =  0 \quad \mbox{in} \: \Omega_{bl} \\
\div v & = 0  \quad \mbox{in} \: \Omega_{bl} \\
v\vert_{\pa \Omega_{bl}} & = \phi.
\end{aligned}
\right.
\end{equation}
already raises difficulties. For instance, to use a strategy based on a transparent boundary condition,  one needs to construct the solution of the  Dirichlet problem in a  half-space for the Stokes-Coriolis operator, when the Dirichlet data has uniform local bounds.  But contrary to the Stokes case,  there is no easy integral representation.  Still, such linear problem was tackled in the recent paper \cite{DalibardPrange}, by the first author and C. Prange.  To solve the Dirichlet problem, they use a Fourier transform in variables $y_1,y_2$, leading to accurate formulas. The point is then to be able to translate information on the Fourier side to uniform local bounds on $v$. This requires careful estimates, as spaces like $L^2_{uloc}$ are defined through truncations in space, which are not so suitable for a Fourier treatment. Similar difficulties arise in paper \cite{ABZ}, devoted to water waves equations in locally uniform spaces.    

\medskip
The linear study \cite{DalibardPrange} is a starting point for our study of the nonlinear system \eqref{BL}, but we  will need many refined estimates, combined with a fixed point argument. More precisely, the outline of the paper is the following.
\begin{itemize}
\item The third and main section of the paper will be devoted to the following system: 
\begin{equation} \label{SC_source}
\left\{
\begin{aligned}
 e \times v  + \na p - \Delta v  & =  {\rm \div} F \inside \{y_3 > M\} \\
\div v & = 0  \inside \{y_3 > M\} \\
v\vert_{y_3 = M} & = v_0.
\end{aligned}
\right.
\end{equation}
The data $v_0$ and  $F$ will have no decay in horizontal variables $(y_1,y_2)$. The source term $F$, which is reminiscent of $u \otimes u$, will decay typically like $|y_3|^{-2/3}$ as $y_3$ goes to infinity. This exponent  is coherent with the decay of $u$ given in Theorem \ref{Thm_WP}.  The point  will be to establish {\it a priori} estimates on a solution $v$ of \eqref{SC_source}, with no decay  in $(y_1,y_2)$, decaying like $|y_3|^{-1/3}$ at infinity.  Functional spaces will be specified in due course.  
\item On the basis of  previous {\it a priori estimates}, we will show well-posedness of the system  
\begin{equation} \label{NSC1}
\left\{
\begin{aligned}
v \cdot \na v + e \times v  + \na p - \Delta v  & = 0 \inside \{y_3 > M\} \\
\div v & = 0  \inside \{y_3 > M\} \\
v\vert_{ \{y_3 = M\} } & = v_0.
\end{aligned}
\right.
\end{equation}
for small enough boundary data $v_0$ (again, in a functional space to be specified). This will be done in the first paragraph of the fourth section. 
\item Finally, through the next paragraphs of the fourth section, we will establish Theorem~\ref{Thm_WP}. The solution $v$ of \eqref{BL} will be constructed with the help of  a mapping ${\cal F} = {\cal F}(\psi,\phi)$, defined  in the following way.  
\begin{enumerate}
\item First, we will introduce the solution $(v^-,p^-)$ of 
\begin{equation} \label{NSC2}
\left\{
\begin{aligned}
v^- \cdot \na v^- +  e \times v^-  + \na p^- - \Delta v^-  & = 0 \inside \Omega_{bl}^M  \\
\div v^- & = 0  \inside \Omega_{bl}^M \\
v^-\vert_{\pa \Omega_{bl}}  & = \phi, \\ 
\Sigma(v^-,p^-) e_3 \vert_{\{y_3 = M\} } & = \psi, 
\end{aligned}
\right.
\end{equation}
where  $\Sigma(v,p)= \na v - \left(p+\frac{|v|^2}{2}\right) \mathrm{Id}$.  Note that a  quadratic term $\frac{|v|^2}{2}$ is added to the usual Newtonian tensor in order to handle the nonlinearity.
\item Then, we will introduce the solution $(v^+,p^+)$ of \eqref{NSC1}, with $v_0 :=  v^-\vert_{y_3 = M}$. 
\item Eventually, we will define ${\cal F}(\psi, \phi) \: := \:  \Sigma(v^+,p^+) e_3  \vert_{\{y_3 = M\} } -\psi$. 
\end{enumerate}
The point will be to show that for small enough $\phi$, the equation ${\cal F}(\psi, \phi) = 0$ has  a solution $\psi$, knowing that ${\cal F}(0,0) = 0$. This will be obtained {\it via} the inverse function theorem (pondering on the linear analysis of \cite{DalibardPrange}).  For such  $\psi$, the field $v$ defined by $v^\pm$ over $\{ \pm y_3 > M\}$ will be a solution of \eqref{BL}. Indeed, $v$ is always continuous at $y_3=M$ by definition of $v^+$, while the condition ${\cal F}(\psi, \phi) =0$ means that the normal component of the stress tensor $\Sigma(v,p) $ is also continuous at $y_3=M$. 
\end{itemize}

\section{Stokes-Coriolis equations with source}
A central part of the work is the analysis of system \eqref{SC_source}. For simplicity, we take $M = 0$. The case without source term ($F =0$) was  partially analyzed in \cite{DalibardPrange}, but we will establish new estimates, notably related to low frequencies. 
Let us emphasize that the difficulty induced by low frequencies already appeared in \cite[Proposition 2.1, page 6]{DalibardPrange}, even in the case of classical Sobolev data: in such case, some cancellation of  the Fourier transform $\hat{v}_{0,3}$  at frequency $\xi = 0$ was assumed.  We make a similar  hypothesis here.  The main theorem of the section is 
\begin{theorem} \label{thm_SCsource}
Let $m \in \N$, $m \gg 1$. Let  $v_0 \in H^{m+1}_{uloc}(\R^2)$ satisfying $v_{0,3} = \pa_1 \nu_1 + \pa_2 \nu_2$, with $\nu_1$, $\nu_2$ in 
$L^2_{uloc}(\R^2)$. Let $F\in H^m_{loc}(\R^3_+)$ such that $(1+y_3)^{2/3}  F \in H^m_{uloc}(\R^3_+)$. There exists a unique solution  $v$ of system \eqref{SC_source} such that 
\begin{multline} \label{theestimate}
\|(1+y_3)^{1/3} v \|_{H^{m+1}_{uloc}(\R^3_+)} \\ 
\le C \left(  \| v_0 \|_{H^{m+\frac{1}{2}}_{uloc}(\R^2)} +  \| (\nu_1, \nu_2) \|_{L^2_{uloc}(\R^2)} + � \|�(1+y_3)^{2/3} F \|_{H^m_{uloc}(\R^3_+)}\right)
\end{multline}
for a universal constant $C$. 
\end{theorem}

Prior to the proof of the theorem, several simplifying remarks are in order: 
\begin{itemize}
\item  Obviously, uniqueness comes down to showing that if $F =0$ and $v_0 = 0$, the only solution $v$ of \eqref{SC_source} such that $(1+y_3)^{1/3} v \in H^m_{uloc}(\R^3_+)$ is  $v =0$. This result follows from \cite[Proposition 2.1]{DalibardPrange}, in which even a larger functional space was considered. Hence, the key statement our theorem is the existence of a solution satisfying the estimate \eqref{theestimate}. 
\item In order to show existence of such a solution, we can assume  $v_{0,1}, v_{0,2}$,  $\nu \defin (\nu_1,\nu_2)$ and $F$ to be smooth and compactly supported (resp. in $\R^2$ and $\R^3_+$). Indeed, let us introduce 
$$\ba
 (v_{0,1}^n, v_{0,2}^n, \nu^n)(y_1,y_2) \: := \:  \chi((y_1,y_2)/n)  \: \rho^n \star (v_{0,1}, v_{0,2}, \nu)(y_1,y_2), \\
  F^n(y)  \: := \:\tilde \chi(y/n) \:  \tilde \rho^n(y) \star F(y)
  \ea$$
where $\chi \in C^\infty_c(\R^2)$, $\tilde \chi \in C^\infty_c(\R^3)$ are $1$ near the origin,  and $\rho^n$, $\tilde \rho^n$ are approximations of unity. These functions are smooth, compactly supported, and satisfy 
\begin{multline*}
 \| (v_{0,1}^n, v_{0,2}^n) \|_{H^{m+1}_{uloc}(\R^2)} \le C  \| (v_{0,1}, v_{0,2}) \|_{H^{m+1}_{uloc}(\R^2)}, \quad    \| \nu^n \|_{H^{m+2}_{uloc}(\R^2)} \le C \| \nu \|_{H^{m+2}_{uloc}(\R^2)}, \\
    \|(1+y_3)^{2/3} F^n \|_{H^m_{uloc}(\R^3_+)} \le C  \|(1+y_3)^{2/3} F \|_{H^m_{uloc}(\R^3_+)}. 
\end{multline*}
for a universal constant $C$. Moreover, $(v_{0,1}^n, v_{0,2}^n)$, $\nu^n$ and $F^n$ converge strongly  to $(v_{0,1}, v_{0,2})$, $\nu$ and $F$ in $H^{m+1}(K)$, $H^{m+2}(K)$ and $H^m(K')$ respectively,  for any compact set $K$ of $\R^2$, any compact set $K'$ of $\R^3_+$.   Now, assume that for all $n\in \N$, there exists a solution $v^n$ corresponding to the data $v_{0,1}^n, v_{0,2}^n,\nu^n, F^n$, for which we can get the estimate 
\begin{multline*}
 \|(1+y_3)^{1/3} v^n \|_{H^{m+1}_{uloc}(\R^3_+)}  \\
  \le C \left(\| (v_{0,1}^n, v_{0,2}^n) \|_{H^{m+1}_{uloc}(\R^2)} +  \| \nu^n \|_{H^{m+2}_{uloc}(\R^2)} +  \|(1+y_3)^{2/3} F^n \|_{H^m_{uloc}(\R^3_+)} \right) 
 \end{multline*}
 for a universal constant $C$. Then, 
\begin{multline*}
 \|(1+y_3)^{1/3} v^n \|_{H^{m+1}_{uloc}(\R^3_+)}  \\ \le C' \left( \| (v_{0,1}, v_{0,2}) \|_{H^{m+1}_{uloc}(\R^2)} + \| \nu \|_{H^{m+2}_{uloc}(\R^2)} +  \|(1+y_3)^{2/3} F \|_{H^m_{uloc}(\R^3_+)} \right) 
 \end{multline*}
 for a universal constant $C'$. 
 We can then extract a  subsequence  weakly converging to some $v$, which is easily seen to satisfy \eqref{SC_source} and \eqref{theestimate}.
 \item Finally, if $v_{0,1}, v_{0,2}$, $\nu$ and $F$ are smooth and compactly supported, the existence of a solution $v$ of \eqref{SC_source} can be obtained by standard variational arguments. More precisely,  one can build a function $v$ such that 
\begin{eqnarray*}
\int_{\R^3_+} | \na v |^2  &\le &C \left( \|F \|_{L^2(\R^2)} \: + \:  \| v_0 \|_{H^{1/2}(\R^2)}\right), \\
 \int_{\R^2 \times \{ y_3 < a \}} | v |^2   &\le& C_a \left( \|F \|_{L^2(\R^2)} +  \| v_0 \|_{H^{1/2}(\R^2)} \right) \:  \quad \forall a > 0. 
 \end{eqnarray*}
 Higher order derivatives are then controlled by elliptic regularity. 
 {\em Hence, the whole problem is to establish the estimate \eqref{theestimate} for such a solution}.   
\end{itemize}
We are now ready to tackle the proof of Theorem \ref{thm_SCsource}. We  forget temporarily about the boundary condition and focus  on the equations 
\begin{equation} \label{SC_eq} 
 e \times v  + \na p - \Delta v   = \div F, \quad \div v = 0 \inside \R^3_+, 
 \end{equation}
Our goal is to construct some particular solution of these equations, satisfying for some large enough $m$: 
\begin{equation} \label{estimate_source}
\|(1+z)^{1/3} v \|_{L^\infty} \le C \:  \|(1+z)^{2/3}  F \|_{L^\infty(H^m_{uloc})}. 
\end{equation}
We will turn to the solution of the whole system  \eqref{SC_source} in a second step.

\subsection{Orr-Sommerfeld formulation} 
To handle \eqref{SC_eq}, we rely on a formulation similar to the Orr-Sommerfeld rewriting of Navier-Stokes. Namely, we wish to express this system  in terms of $v_3$ and $\omega \defin \pa_1 v_2 - \pa_2 v_1$. First, we  apply  $\pa_2$ to the first line, $-\pa_1$ to the second line, and combine to obtain 
\begin{equation} \label{OS1}
 \pa_3 v_3 + \Delta \omega = s_3 \defin \pa_2 f_1 - \pa_1 f_2, \quad \mbox{with} \:  f \defin \div F = (\sum_j \pa_j F_{ij})_{i}. 
  \end{equation}
Similarly, we apply  $\pa_1 \pa_3$ to the first line of \eqref{SC_eq}, $\pa_2 \pa_3$ to the second line, and $-(\pa_1^2 + \pa_2^2)$ to the third line. Combining the three, we are left with 
\begin{equation} \label{OS2}
 -\pa_3 \omega + \Delta^2 v_3  = s_\omega  \defin \pa_1 \pa_3 f_1 + \pa_2 \pa_3 f_2 -(\pa_1^2 + \pa_2^2) f_3.
 \end{equation}
From $\omega$ and $v_3$,   one  recovers the horizontal velocity components $v_1, v_2$ using the system 
 $$ \pa_1 v_1 + \pa_2 v_2 = - \pa_3 v_3,  \quad  \pa_1 v_2 - \pa_2 v_1 = \omega. $$

We are led to the (so far formal) expressions 
 \begin{equation} \label{formulav1v2}
 \begin{aligned}
  v_1 & =(\pa_1^2+\pa_2^2)^{-1} \left(- \pa_3 \pa_1 v_3 - \pa_2 \omega \right) \\
 v_2 & = (\pa_1^2+\pa_2^2)^{-1} \left( -\pa_3 \pa_2 v_3 + \pa_1 \omega \right).
 \end{aligned}
 \end{equation}
Our goal is to construct a solution $(v_3, \omega)$  of \eqref{OS1}-\eqref{OS2},  by means of an integral representation. Since the vertical variable will play a special role in this construction, we will denote it by $z$ instead of $y_3$: $y = (y_1,y_2,z)$. We write \eqref{OS1}-\eqref{OS2} in the compact form 
$$ L(D,\pa_z) V = S,  \quad V \defin \left( \begin{smallmatrix} v_3 \\   \omega \end{smallmatrix} \right), \quad S \defin \left( \begin{smallmatrix}   s_3\\ s_\om \end{smallmatrix} \right), \quad  D \defin \frac{1}{i} (\pa_1,\pa_2), $$
where $L(D,\pa_z)$ is a Fourier multiplier in variables $x_1,x_2$ associated with 
$$ L(\xi, \pa_z) \defin \left( \begin{matrix} \pa_z & (\pa^2_z - |\xi|^2)\\ (\pa^2_z - |\xi|^2)^2 & -\pa_z  \end{matrix} \right). $$ 
We will look for a solution under the form 
\begin{equation} \label{integralV}
 V(\cdot,z)  = \int_0^{+\infty} G(D,z-z')  S(\cdot, z') \, dz' \: + \: V_h 
 \end{equation}
where 
\begin{itemize}
\item $G(D,z)$ is a matrix Fourier multiplier, whose symbol $G(\xi,z)$ is the fundamental solution over $\R$ of  $L(\xi,\pa_z)$ for any $\xi \in \R^2$:
$$ L(\xi,\pa_z) G(\xi,z) = \delta_{z=0} \,  \left( \begin{smallmatrix} 1 & 0 \\ 0 & 1   \end{smallmatrix} \right). $$
\item $V_h$ is a solution of the homogeneous equation. The purpose of the addition of  $V_h$ is to ensure the decay of the solution $V$. More details will be given in due course. 
\end{itemize}
\subsubsection{Construction of the Green function.}
We start with the construction of the fundamental solution $G(\xi,z)$. Away from $z=0$, it should satisfy the homogeneous system, which requires to understand the kernel of the operator $L(\xi,\pa_z)$. This kernel is a combination  of elements of the form $e^{\lambda z} V$, where 
$\lambda$ is a root of the characteristic equation
\begin{equation} \label{eqn_roots}
\det L(\xi,\lambda) = 0, \quad \mbox{\it i.e.} \:  - \lambda^2 -(\lambda^2-|\xi|^2)^3= 0, 
\end{equation}
and $V$ an associated  ``eigenelement", meaning a non-zero vector in  $\ker L(\xi,\lambda)$. A careful study of the characteristic equation was carried recently by the first author and C. Prange in \cite{DalibardPrange}.  Notice that \eqref{eqn_roots} can be seen as an equation of degree three on $Y= \lambda^2-|\xi|^2$  (with negative discriminant). Using Cardano's formula gives access to explicit expressions. The roots can be written $\pm \lambda_1(\xi)$, $\pm \lambda_2(\xi)$ and $\pm \lambda_3(\xi)$, where $\lambda_1 \in \R_+$, $\lambda_2,\lambda_3$ have positive real parts,  $\lambda_1 \in \R$, $\overline{\lambda_2} = \lambda_3$, $\textrm{Im} \lambda_2 > 0$.  The $\lambda_i$'s are continuous functions of $\xi$ (see Remark \ref{remark_roots} below for more). 
Article \cite{DalibardPrange} also provides their asymptotic behaviour at low and high frequencies. This behaviour will be very important to establish our estimates.
\begin{lemma} {\bf (\cite[Lemma 2.4]{DalibardPrange})} \label{asymptote_lambda}

As $\xi \rightarrow 0$, we have
\begin{equation*}
\lambda_1(\xi)  = |\xi|^3 + O(|\xi|^5), \quad
\lambda_2(\xi)  = e^{i \frac{\pi}{4}}+ O(|\xi|^2), \quad
\lambda_3(\xi)  = e^{-i \frac{\pi}{4}}+ O(|\xi|^2), \quad
\end{equation*}

As $\xi \rightarrow \infty$, we have 
\begin{multline*}
\lambda_1(\xi)  = |\xi| - \frac{1}{2} |\xi|^{-1/3} + O(|\xi|^{-5/3}), \quad 
\lambda_2(\xi)  = |\xi| - \frac{j^2}{2} |\xi|^{-1/3} + O(|\xi|^{-5/3}), \\
\lambda_3(\xi)  = |\xi| - \frac{j}{2} |\xi|^{-1/3} + O(|\xi|^{-5/3}),  \quad \mbox{where $j = \exp(2i\pi/3)$}. 
\end{multline*}
\end{lemma}
\begin{remark}  \label{remark_roots}
We insist that $\lambda_2$ et $\lambda_3$ are distinct and have a  positive real part for all values of  $\xi$, whereas $\lambda_1 \neq 0 $  for $\xi \neq  0$. Moreover, it can be easily checked that $\lambda_i^2$ is a $\mathcal C^\infty$ function of $|\xi|^2$ for $i=1..3$. Using the fact that $\lambda_2$ and $\lambda_3$ never vanish or merge, while $\lambda_1$ vanishes for $\xi=0$ only, we deduce that $\lambda_2, \lambda_3$ are $\mathcal C^\infty $ functions of $|\xi|^2$, and that $\lambda_1(\xi)= |\xi|^3 \Lambda_1(\xi)$, where $\Lambda_1\in \mathcal C^\infty(\R^2)$, $\Lambda_1(0)=1$ and $\Lambda_1$ does not vanish on $\R^2$.
 \end{remark}
As regards the eigenelements, an explicit computation shows that for all $i=1...3$, 
\begin{equation} \label{eigenelements}
 V_i^\pm \defin \left( \begin{smallmatrix} 1 \\ \pm \Omega_i \end{smallmatrix} \right), \quad \Omega_i \defin \frac{-\lambda_i}{\lambda_i^2 - |\xi|^2}  \quad  \mbox{satisfy    $L(\xi,\pm \lambda_i) V_i^\pm  = 0$.} 
 \end{equation}
 
 We can now determine $G$; our results are summarized in Lemma \ref{lem:G} below. We begin with its first column $G_1 = \left( \begin{smallmatrix} G_{11} \\ G_{21} \end{smallmatrix} \right)$, solution of $L(\xi,\pa_z) G_1  = \delta \,  \left( \begin{smallmatrix}  1 \\ 0 \end{smallmatrix} \right)$. 
As explained above, for $z\neq 0$, $G_1(\xi, z)$ is a linear combination  of $e^{\pm\lambda_i z} V_i^\pm$. Furthermore, we want to avoid any exponential growth of $G$ as $z\to \pm \infty$.
 Thus $G_1$ should be of the form 
 \begin{equation*}
  G_1 = \left\{
 \begin{aligned}
 & \sum_{i=1}^3  A_i^+ e^{-\lambda_i z} V_i^-, \quad z > 0, \\
& \sum_{i=1}^3  A_i^- e^{\lambda_i z} V_i^+, \quad z < 0.
\end{aligned}
\right.
\end{equation*}
We now look at the jump conditions at $z=0$. We recall that for $f = f(z)$, $[f]\vert_{z=z'} \defin f(z^{'+}) - f(z^{'-})$ denotes the jump of $f$ at $z'$. Since
\begin{equation*}
\left\{
\begin{aligned}
(\pa_z^2 -|\xi|^2)^2 G_{11} - \pa_z G_{21} & = 0, \\
\pa_z G_{11} + (\pa_z^2 - |\xi|^2) G_{21} & = \delta_{z=0}, 
\end{aligned}
\right.
\end{equation*}
we infer that 
\begin{equation*}
[G_{21}]\vert_{z=0} = 0,\quad  [\pa_z G_{21}]\vert_{z=0} = 1, \quad [\pa^k_z G_{11}]\vert_{z=0} = 0, \: k=0...3, \\
\end{equation*} 
This yields a linear system of 6 equations on the coefficients $A_i^\pm$. One finds $A_i \defin A_i^+ = - A_i^-$, and the system
\begin{equation*}
 \sum_i \lambda_i \Omega_i A_i  = \frac{1}{2}, \quad \sum _i A_i = 0, \quad \sum_i \lambda_i^2 A_i = 0 .
 \end{equation*}
 Note that 
 $$ \sum_i \lambda_i \Omega_i A_i = -\sum_i \frac{\lambda_i^2}{\lambda_i^2 - |\xi|^2} A_i = - \sum_i \frac{|\xi|^2}{\lambda_i^2 - |\xi|^2} A_i $$
 taking into account the second equality. Hence, we find 
 $$ \begin{pmatrix}  \frac{|\xi|^2}{\lambda_1^2 - |\xi|^2} &  \frac{|\xi|^2}{\lambda_2^2-|\xi|^2} &  \frac{|\xi|^2}{\lambda_3^2-|\xi|^2} \\1 & 1 & 1 \\ \lambda_1^2 & \lambda_2^2 & \lambda_3^2 \end{pmatrix}  \begin{pmatrix} A_1 \\ A_2  \\ A_3 \end{pmatrix} = \begin{pmatrix} -\frac{1}{2} \\ 0  \\0 \end{pmatrix}.$$
 The determinant of the matrix is
 $$ D_1 \defin |\xi|^2 D,$$
 where
 $$
 D\defin \left| \begin{array}{ccc}
\frac{1}{\lambda_1^2 - |\xi|^2} &  \frac{1}{\lambda_2^2-|\xi|^2} &  \frac{1}{\lambda_3^2-|\xi|^2} \\1 & 1 & 1 \\ \lambda_1^2 & \lambda_2^2 & \lambda_3^2
 \end{array}\right|.
 $$
 After a few  computations, we find that
 \begin{equation} \label{D1}
 D_1 = |\xi|^2 (\lambda_2^2 - \lambda_1^2)(\lambda_3^2 - \lambda_1^2) \left( \frac{1}{(\lambda_1^2 - |\xi|^2)(\lambda_2^2 - |\xi|^2)} -  
 \frac{1}{(\lambda_1^2 - |\xi|^2)(\lambda_3^2 - |\xi|^2)} \right), 
 \end{equation}
 and 
 \begin{equation} \label{Ai}
 A_1 = -\frac{1}{2 D_1} (\lambda_3^2 - \lambda_2^2), \quad A_2 = -\frac{1}{2D_1} (\lambda_1^2 - \lambda_3^2), \quad A_3 = -\frac{1}{2D_1}  (\lambda_2^2 - \lambda_1^2).
\end{equation}
%

\medskip
Computations for the second column $G_2$ of $G$ are similar. It is of the form 
 \begin{equation*}
  G_2 = \left\{
 \begin{aligned}
 & \sum_{i=1}^3  B_i^+ e^{-\lambda_i z} V_i^-, \quad z > 0, \\
& \sum_{i=1}^3  B_i^- e^{\lambda_i z}  V_i^+, \quad z < 0,
\end{aligned}
\right.
\end{equation*}
 with jump conditions:
$$\mbox{$[\pa^k_z G_{22}]\vert_{z=0} = 0$, $k=0,1$, $\: [\pa^k_z G_{12}]\vert_{z=0} = 0$, $k=0...2$, $[\pa^3_z G_{12}]\vert_{z=0} = 1$.}$$
 We find $B_i \defin B_i^+ = B_i^-$, and the system 
 $$ \begin{pmatrix}  \Omega_1 &  \Omega_2 &  \Omega_3 \\\lambda_1 & \lambda_2 &  \lambda_3 \\ \lambda_1^3 & \lambda_2^3 & \lambda_3^3 \end{pmatrix}  \begin{pmatrix} B_1 \\ B_2  \\ B_3 \end{pmatrix} = \begin{pmatrix}  0  \\0 \\ -\frac{1}{2} \end{pmatrix}.$$
The determinant of the matrix is now $D_2\defin  -\lambda_1 \lambda_2 \lambda_3 D$, and 
 \begin{equation} \label{Bi}
 \begin{aligned}
 B_1 & = \frac{\lambda_2 \lambda_3}{2 D_2} \left(\frac{1}{\lambda_2^2 - |\xi|^2}  -  \frac{1}{\lambda_3^2 - |\xi|^2} \right), \quad B_2 = \frac{\lambda_1 \lambda_3}{2D_2}  \left(\frac{1}{\lambda_3^2 - |\xi|^2}  -  \frac{1}{\lambda_1^2 - |\xi|^2}\right)  , \\
  B_3 & = \frac{\lambda_1 \lambda_2}{2D_2} \left(\frac{1}{\lambda_1^2 - |\xi|^2}  -  \frac{1}{\lambda_2^2 - |\xi|^2}\right).
 \end{aligned}
 \end{equation}
This concludes the construction of the matrix $G$.  We sum up our results in the following Lemma, in which we also give the asymptotic behaviours of the coefficients $A_i, B_i, V_i^\pm$ and of $G$ as $\xi\to 0$ and $|\xi|\to \infty$. The latter follow from Lemma \ref{asymptote_lambda} and Remark \ref{remark_roots} and are left to the reader.
 \begin{lemma}\label{lem:G}
We have 
\begin{equation*} \label{G1}
  G_1 = \left\{
 \begin{aligned}
 & \sum_{i=1}^3  A_i e^{-\lambda_i z} V^-_i, \quad z > 0, \\
& -\sum_{i=1}^3  A_i e^{\lambda_i z} V^+_i, \quad z < 0,
\end{aligned}
\right.\quad G_2 = \left\{
 \begin{aligned}
 & \sum_{i=1}^3  B_i e^{-\lambda_i z} V_i^-, \quad z > 0, \\
& \sum_{i=1}^3  B_i e^{\lambda_i z} V^+_i, \quad z < 0,
\end{aligned}
\right.
\end{equation*}
where $V_i^\pm= \begin{pmatrix}
1\\ \mp \frac{\lambda_i}{\lambda_i^2 - |\xi|^2}
\end{pmatrix}$ and where $A_i$ and $B_i$ are defined by \eqref{Ai}and  \eqref{Bi} respectively.

\medskip
\noindent
Asymptotic behaviour:
\begin{itemize}
\item For $|\xi|\gg 1$, there exists $N>0$ such that $A_i, B_i, \Omega_i= O(|\xi|^N)$ for $i=1..3$, and $|\Om_i| \gtrsim |\xi|^{-N}$. As a consequence, $G(\xi, z) = O(|\xi|^N)$ for all $z$.

\item As $\xi \to 0$, we have
\begin{equation}\ba
A_i(\xi)\to \bar A_i \in \C^*,\ i=1..3,\ B_1(\xi)\sim \frac{\bar B_1}{|\xi|},\ \bar B_1\in \C^*,\ B_i(\xi)\to \bar B_i \in \C^* \ i=2,3,\\
\Om_1\sim \bar \Om_1 |\xi|,\ \bar \Om_1\in\C^*,\  \Om_i(\xi)\to \bar \Om_i \in \C^*, i=2,3.
\ea
 \end{equation}
 More precisely, we can write for instance
 $$
 B_1(\xi)= \frac{\bar B_1}{|\xi|} \beta_1(\xi)\quad \forall \xi \in \R^2,
 $$
 for some function $\beta_1\in \mathcal C^\infty(\R^2)$ such that $\beta_1(0)=1$. Similar statements hold for the other coefficients.
 
It follows that $G(\xi,z) = \begin{pmatrix} O(1) & O(|\xi|^{-1}) &  \\ O(1) & O(1) \end{pmatrix}$ as $|\xi| \rightarrow 0$ for all $z\in \R$.
\end{itemize}

 \end{lemma}

\subsubsection{Construction of the homogeneous correction}
We will see rigorously below that the field 
\begin{equation} \label{VG}
 V_G(\cdot,z) \defin   \int_0^{+\infty} G(D,z-z')  S(\cdot, z') \, dz' =   \int_0^{+\infty} {\cal F}_{\xi \rightarrow (y_1,y_2)}^{-1} \left( G(\cdot,z-z')   {\cal F}_{ (y_1,y_2)\rightarrow \xi } S(\cdot, z') \right) \, dz' 
 \end{equation}
is well-defined and satisfies  \eqref{OS1}-\eqref{OS2}. However, the corresponding velocity field does not have  a good decay with respect to $z$. This is the reason for the additional field $V_h$ in formula \eqref{integralV}. To be more specific, let us split the source term $S$ into $S(z') = S^0(z') + \pa_{z'} S^1(z') + \pa_{z'}^2 S^2(z')$, with 
\be\label{def:S0} S^0(z') \defin \begin{pmatrix} \pa_2 (\pa_1 F_{11} + \pa_2 F_{12}) -  \pa_1 (\pa_1 F_{21} + \pa_2 F_{22})  \\
- (\pa_1^2 + \pa_2^2) (\pa_1 F_{31} + \pa_2 F_{32}) \end{pmatrix}\ee
and
\be\label{def:S12}\begin{aligned}
S^1(z') & \defin    \begin{pmatrix} \pa_2 F_{13} - \pa_1 F_{23}   \\  \pa_1 (\pa_1 F_{11} + \pa_2 F_{12})  + \pa_2 (\pa_1 F_{21} + \pa_2 F_{22}) - (\pa_1^2 + \pa^2_2)F_{33} \end{pmatrix}, \\
S^2(z') & \defin \begin{pmatrix} 0 \\\pa_1 F_{13} + \pa_2 F_{23} \end{pmatrix}. 
\end{aligned} \ee
Roughly, the idea is that 
$$ V(\cdot,z) :=  \int_0^{+\infty} \left( G(D,z-z') S^0(z') + \pa_{z}G(D,z-z') S^1(z') + \pa^2_{z}G(D,z-z') S^2(z') \right) \, dz' $$
has  a better decay. Using the fact that $\pa_z G(D,z-z') = -\pa_z' G(D,z-z')$, we see that going from $V_G$ to $V$ is possible through integrations by parts in variable $z'$, which generates boundary terms. We recall that the jump of $G(D,z-z')$ at $z=z'$ is zero, and that 
$$
\left[\pa_z G(D,z-z')\right]\vert_{z=z'}=\begin{pmatrix}
0&0\\1&0
\end{pmatrix}.
$$
On the other hand, the first component of $S^2$ is zero, so that the jump of $\pa_z G_{21}$ at $z = z'$ is not involved in the two integrations by part of $\pa^2_{z}G(D,z-z') S^2(z')$.
Formal computations eventually lead to
\begin{align*} 
V_h(\cdot,z) & \defin V(\cdot,z) - V_G(\cdot,z) \\
 & \hspace{0.1cm}  = \: - \Bigl[ G(D,z-z') (S^1(\cdot,z') + \pa_z S^2(\cdot,z') \Bigr]_0^{+\infty}  \: + \:  \Bigl[ \pa_z G(D,z-z') S^2(\cdot,z') \Bigr]_0^{+\infty} \\
 &  \hspace{0.1cm} =  \:  G(D,z) (S^1(\cdot,0) + \pa_{z'} S^2(\cdot,0) )  \: - \:   \pa_z G(D,z)  S^2(\cdot,0).
 \end{align*}
 Back to the expression of the Green function, we get
\begin{equation} \label{Vh}
\begin{aligned}
V_h(\cdot,z) & =   -\: \begin{pmatrix} \sum_i A_i e^{-\lambda_i z} V_i^- & \sum_i B_i e^{-\lambda_i z} V_i^- \end{pmatrix} (S^1(\cdot,0)  + \pa_{z'} S^2(\cdot,0) ) \\ 
&  + \:  \begin{pmatrix} \sum_i A_i \lambda_i e^{-\lambda_i z} V_i^- & \sum_i B_i \lambda_i e^{-\lambda_i z} V_i^- \end{pmatrix} S^2(\cdot,0). 
\end{aligned}
\end{equation}
It is  a linear  combination  of terms of the form   $e^{-\lambda_i z} V_i^-$, and therefore satisfies the homogeneous Orr-Sommerfeld equations.  Hence, $V$ is (still formally) a solution of \eqref{OS1}-\eqref{OS2}. 

\medskip
We now need to put these  formal arguments  on rigorous grounds. As mentioned after Theorem \ref{thm_SCsource}, there is no loss of generality assuming that $F$ is smooth and compactly supported.  We first claim 
\begin{lemma} \label{rigorous}
Let  $F$ smooth and compactly supported. The formula \eqref{integralV}, with $V_h$ given by \eqref{Vh}, defines a solution 
$V = (v_3,\omega)^t$ of \eqref{OS1}-\eqref{OS2} satisfying
$$ V \in L^\infty_{loc}(\R_+, H^m(\R^2)), \quad |D|^{-1} \omega \in L^\infty_{loc}(\R_+, H^m(\R^2)) \quad \mbox{for any $m$.}$$  
\end{lemma}

\medskip
\begin{proof}

Let us show first show that the integral term $V_G$ (see \eqref{VG}) satisfies the properties of the lemma. The main point is to show  that for any $z,z' \ge  0$,  the function 
$$J_{z,z'} : \xi  \rightarrow G(\xi,z-z') \hat{S}(\xi,z') \quad  \mbox{belongs to} \quad
 L^2((1+|\xi|^2)^{m/2} d\xi)\times  L^2(|\xi|^{-1}(1+|\xi|^2)^{m/2} d\xi)$$
for all $m$. Therefore, we recall  that $\hat{F} = \hat{F}(\xi,z')$ is in the Schwartz class with respect to $\xi$,  smooth and compactly supported in $z'$. Also, $G(\xi,z-z')$ is smooth in $\xi \neq 0$ (see Remark \ref{remark_roots}), and continuous in $z,z'$.  It implies that $J_{z,z'}$ is smooth in $\xi \neq 0$, continuous in $z,z'$. It remains to check its behaviour  at high and low frequencies. 
\begin{itemize}
\item At high frequencies ($|\xi| \gg 1$), from Lemma \ref{lem:G}, it is easily seen that $J_{z,z'}$ is bounded by 
$$ |J_{z,z'}(\xi)| \le C |\xi|^N \sum_{k=0}^2 |\pa^k_{z'} \hat{F}(\xi,z')| $$
for some $N$. As $\hat{F}$ and its $z'$-derivatives are rapidly decreasing in $\xi$, it will belong to any $L^2$ with polynomial weight. 
\item At low frequencies ($\xi \sim 0$),  one can check that  
$|\hat{S}(\xi,z')| \le C |\xi|$.   Hence, using again the bounds derived in Lemma \ref{lem:G}, $G(\xi,z-z') \hat{S}(\xi,z') = \begin{pmatrix} O(1) \\O(|\xi|)  \end{pmatrix}$. The result follows.
\end{itemize}  
From there, by standard arguments, $V_G$ defines a continuous function of $z$ with values in $H^m(\R^2) \times |D|^{-1} H^m(\R^2)$ for all $m$. 
Moreover, a change of variable gives
$$ V_G(\cdot,z) = \int_0^{+\infty} G(D,z') S(\cdot,z-z') dz'.   $$
By the smoothness of $S$, we deduce  that $V_G$ is smooth in $z$ with values in the same space. The fact that it satisfies \eqref{OS1}-\eqref{OS2} comes of course from the properties of the Green function $G$, and is classical. We leave it to the reader.

\medskip
To conclude the proof of the lemma, we still have to consider the homogeneous correction $V_h$.  Again,  $V_h$ is smooth in $\xi \neq 0$ and $z$. Thanks to the properties of $F$, it is  decaying fast  as $|\xi|$ goes to infinity. Moreover, from the asymptotics above, one can check that $V_h = \left( \begin{smallmatrix}  O(1) \\O(|\xi|) \end{smallmatrix} \right)$ for $|\xi|\ll 1$. Finally, as its Fourier transform is a linear combination of 
$e^{-\lambda_i (\xi)z} V^-_i(\xi)$, it satisfies  \eqref{OS1}-\eqref{OS2} without source.  This ends the proof. 
\end{proof} 
\medskip
Let us stress that, with the same kind of arguments, one can  justify the integration by parts mentioned above, and write 
\begin{equation} \label{integralVbis}
V(\cdot,z) :=  \int_0^{+\infty} \sum_{k=0}^2 \pa_z^k G(D,z-z') S^k(z')  \, dz' 
\end{equation}
We will now try to derive the estimate \eqref{estimate_source}, starting from this formulation.

\subsubsection{Main estimate}
By Lemma \ref{rigorous}, we know that formula \eqref{integralV} (or equivalently \eqref{integralVbis}) defines a solution $V$ of \eqref{OS1}-\eqref{OS2}. Our main goal in this paragraph is to establish that $V$ obeys  inequality \eqref{estimate_source}. Our main ingredient will be 
\begin{lemma} \label{lemmadecay}
Let $\chi = \chi(\xi) \in C^\infty_c(\R^2)$, and $P = P(\xi) \in C^\infty(\R^2\setminus\{0\})$ defined by
$$ P(\xi) = p_k (\xi)|\xi|^{\alpha - k} Q(\xi) $$
near $\xi = 0$,  with $p_k$  a  homogeneous polynomial in $\xi_1, \xi_2$ of  degree $k$, $\alpha> 0$, and $Q \in \mathcal C^\infty(\R^2)$. Assume furthermore that $\alpha-k\ge -2$.
For $u_0 \in L^1_{uloc}(\R^2)$,we define $u^i = u^i(y_1,y_2,z)$  by  
\begin{equation} \label{def_ui_low}
 u^i(\cdot,z)  \defin \chi(D) P(D) e^{-\lambda_i(D) z} u_0. 
 \end{equation}
Then, there exists  $C$ and $\delta > 0$ independent of $u_0$ such that 
$$   \| e^{\delta z}  u^2 \|_{L^\infty(\R^3_+)} +  \| e^{\delta z}  u^3 \|_{L^\infty(\R^3_+)}  \| \: \le \: C \|�u_0 \|_{L^1_{uloc}}.$$ 
Moreover, there exists  $C$ and $\delta > 0$ independent of $u_0$ such that 
$$\| (1+z)^{\frac{\alpha}{3}} u^1 \|_{L^\infty(\R^3_+)}  \| \: \le \: C \|u_0 \|_{L^1_{uloc}}.$$ 
\end{lemma} 
\begin{remark}
Showing that the definition \eqref{def_ui_low} makes sense is part of the proof of the lemma. Namely, it is shown  that for any $z > 0$,  the kernel 
$$ K(x_1,x_2,z) \: :=  \: {\cal F}^{-1}_{\xi \rightarrow (x_1,x_2)} \left( \chi(\xi) P(\xi) e^{-\lambda_i(\xi) z} \right)$$ 
defines an element of $L^1(\R^2)$. In particular,  \eqref{def_ui_low} is appropriate:  
$u^i = K(\cdot, z) \star  u_0$ defines (at least) an $L^1_{uloc}$ function as the convolution of  functions of $L^1$ and  $L^1_{uloc}$. 
\end{remark}
We refer to the appendix for a proof. Lemma \ref{lemmadecay} is the source of  the asymptotic behaviour of the solution $v$ of  \eqref{BL}. As always in this type of boundary layer problems, the asymptotic behaviour is given by low frequencies, corresponding to the cut-off $\chi$. In particular, the decay is given by the characteristic root $\lambda_1(\xi)$, which vanishes at $\xi = 0$.

\medskip
{\em Proof of estimate \ref{estimate_source}.} 
We distinguish between low and high frequencies.

 \medskip
 \noindent
 {\em Low frequencies.}
We introduce some $\chi = \chi(\xi) \in C^\infty_c(\R^2)$ equal to $1$ near $\xi=0$. 
We consider 
 \begin{equation} \label{def_vbemol}
  V^{\flat} = \int_{\R^+}  \sum_{k=0}^2 I^k(\cdot,z,z') dz', \quad   I^k(\cdot,z,z')  \defin  \chi(D)  \pa^k_{z} G(D,z-z') S^k(\cdot,z').
 \end{equation} 
 In what follows,  we write  $S^k  = (s^k_3, s^k_\omega )^t$  and $I^k = (I^k_3, I^k_\omega)^t$. 
 \begin{itemize}
\item  Study of $I^0$.
\end{itemize}
We find 
\begin{align*}
 I^0_3(\cdot,z,z')   \:  =& \:  \sgn(z-z')  \chi(D)   \sum A_i(D) e^{-\lambda_i(D)|z -z'|}   s^0_3(\cdot,z')  
 \\
 & + \: \chi(D)  \sum B_i(D) e^{-\lambda_i(D)|z -z'|} s^0_\om(\cdot,z'),  \\
 I^0_\omega(\cdot,z,z')  \:  =& \:  -\chi(D)   \sum A_i(D) \Omega_i(D) e^{-\lambda_i(D)|z -z'|}   s^0_3(\cdot,z')  \\
&   -\sgn(z-z') \:  \chi(D) \sum B_i(D) \Omega_i(D) e^{-\lambda_i(D)|z -z'|} s^0_\om(\cdot,z').  
 \end{align*}
We also have 
 \begin{align*}
  \pa_z  I^0_3(\cdot,z,z')   \: = \:  &  -\: \chi(D)   \sum A_i(D) \lambda_i(D) e^{-\lambda_i(D)|z -z'|}   s^0_3(\cdot,z')  \\
& -\sgn(z-z')   \: \chi(D)  \sum B_i(D) \lambda_i(D) e^{-\lambda_i(D)|z -z'|} s^0_\om(\cdot,z'). 
\end{align*}
 We note that $\widehat{s^0_3}(\xi,z')$, resp. $\widehat{s^0_\om}(\xi,z')$ is a product of components of $\hat{F}(\xi,z')$ by  homogeneous polynomials of  degree 2, resp. degree 3 in $\xi$.  Using the asymptotic behaviours derived in Lemma \ref{lem:G} together with Lemma \ref{lemmadecay},  we deduce
 \begin{equation} \label{I0}
 \begin{aligned}
 \|  I^0_3(\cdot, z,z') \|_{L^\infty(\R^2)} & \le  \frac{C}{(1+|z-z'|)^{2/3}} \, \| F(\cdot, z') \|_{L^1_{uloc}(\R^2)}, \\   
 \|  I^0_\omega(\cdot, z,z') \|_{L^\infty(\R^2)} & \le  \frac{C}{(1+|z-z'|)}  \| F(\cdot, z') \|_{L^1_{uloc}(\R^2)}, \\  
\left\|  \frac{D}{|D|^2} I^0_\omega(\cdot,z,z') \right\|_{L^\infty(\R^2)}   & \le  \frac{C}{(1+|z-z'|^{2/3})}  \| F(\cdot, z') \|_{L^1_{uloc}(\R^2)} \\ 
 \left\|  \frac{D}{|D|^2}  \pa_z I^0_3(\cdot,z,z') \right\|_{L^\infty(\R^2)} & \le  \frac{C}{(1+|z-z'|^{4/3})}  \| F(\cdot, z') \|_{L^1_{uloc}(\R^2)}
\end{aligned}
\end{equation}
 The last two bounds will be useful when estimating the horizontal velocity components through \eqref{formulav1v2}. We insist that $\pa_z I^0_3$ has a better behaviour than $I^0_3$, because there is an extra factor $\lambda_1(D)$ in front of $A_1$ and $B_1$, which  gives a higher degree of homogeneity at low frequencies for the term in $\exp(-\lambda_1(D) z)$. This is why we can apply $D/|D|^2$ to that term. As for the terms in $\exp(-\lambda_i(D)z)$ for $i=2,3$, there is no singularity near $\xi=0$ when we apply $D/|D|^2$ because of the homogeneity of degree 2 - resp. 3 - in $\widehat{s^0_3}(\xi,z')$ - resp. $\widehat{s^0_\om}(\xi,z')$. 

\begin{itemize}
\item Study of $I^1$.
\end{itemize}
 We find 
  \begin{align*}
 I^1_3(\cdot,z,z')   \: = \: & - \chi(D)   \sum A_i(D) \lambda_i(D) e^{-\lambda_i(D)|z -z'|}   s^1_3(\cdot,z')  
 \\
   & -\sgn(z-z') \: \chi(D)  \sum B_i(D) \lambda_i(D) e^{-\lambda_i(D)|z -z'|} s^1_\om(\cdot,z'),  \\
 I^1_\omega(\cdot,z,z')  \: = \: &  \sgn(z-z')  \chi(D)   \sum A_i(D) \lambda_i(D) \Omega_i(D) e^{-\lambda_i(D)|z -z'|}   s^1_3(\cdot,z')  \\
&  + \: \chi(D)  \sum B_i(D) \lambda_i(D)  \Omega_i(D) e^{-\lambda_i(D)|z -z'|} s^1_\om(\cdot,z'), 
 \end{align*}
and also 
\begin{align*}
\pa_z I^1_3(\cdot,z,z')   \: = \: & \sgn(z-z')  \chi(D)   \sum A_i(D) (\lambda_i(D))^2 e^{-\lambda_i(D)|z -z'|}   s^1_3(\cdot,z')  
 \\
&+\: \chi(D)  \sum B_i(D) (\lambda_i(D))^2 e^{-\lambda_i(D)|z -z'|} s^1_\om(\cdot,z').  \\
\end{align*}
 Thanks to the derivation of the Green function with respect to $z$, an extra factor $\lambda_1(D)$ appears together with $A_1(D)$ or $B_1(D)$. 
 This provides a higher degree of homogeneity in $|\xi|$ at low frequencies. It compensates for the loss of homogeneity of $S^1$ compared to $S^0$. More precisely, we note that $\widehat{s^1_3}(\xi,z')$, resp. $\widehat{s^1_\om}(\xi,z')$ is a product of components of $\hat{F}(\xi,z')$ by  homogeneous polynomials of  degree 1, resp. degree 2 in $\xi$. 
 We also get 
 \begin{equation} \label{I1}
 \begin{aligned}
\|  I^1_3(\cdot, z,z') \|_{L^\infty(\R^2)} & \le  \frac{C}{(1+|z-z'|)^{4/3}} \, \| F(\cdot, z') \|_{L^1_{uloc}(\R^2)}, \\   
\|  I^1_\omega(\cdot, z,z') \|_{L^\infty(\R^2)} & \le  \frac{C}{(1+|z-z'|^{5/3})}  \| F(\cdot, z') \|_{L^1_{uloc}(\R^2)}, \\
\left\|  \frac{D}{|D|^2} I^1_\omega(\cdot,z,z') \right\|_{L^\infty(\R^2)} & \le  \frac{C}{(1+|z-z'|^{4/3})}  \| F(\cdot, z') \|_{L^1_{uloc}(\R^2)}, \\
\left\|  \frac{D}{|D|^2}  \pa_z I^1_3(\cdot,z,z') \right\|_{L^\infty(\R^2)} & \le  \frac{C}{(1+|z-z'|^2)}  \| F(\cdot, z') \|_{L^1_{uloc}(\R^2)}.
\end{aligned}
\end{equation}
 \begin{itemize}
 \item Study of $I^2$.
 \end{itemize}
 We find 
\begin{align*}
 I^2_3(\cdot,z,z')   \: = \:  & \sgn(z-z')  \chi(D)   \sum A_i(D) (\lambda_i(D))^2 e^{-\lambda_i(D)|z -z'|}   s^2_3(\cdot,z')  
 \\
   & +  \: \chi(D)  \sum B_i(D) (\lambda_i(D))^2 e^{-\lambda_i(D)|z -z'|} s^2_\om(\cdot,z'),  
\end{align*}
as well as 
\begin{align*}
 I^2_\omega(\cdot,z,z')  \:  = \: &-  \chi(D)   \sum A_i(D) (\lambda_i(D))^2\Omega_i(D) e^{-\lambda_i(D)|z -z'|}   s^2_3(\cdot,z')  \\
& -\sgn(z-z') \: \chi(D)  \sum B_i(D) (\lambda_i(D))^2  \Omega_i(D) e^{-\lambda_i(D)|z -z'|} s^2_\om(\cdot,z'),
 \end{align*}
and 
\begin{align*}
 \pa_z I^2_3(\cdot,z,z')   \: = \:  & - \chi(D)   \sum A_i(D) (\lambda_i(D))^3 e^{-\lambda_i(D)|z -z'|}   s^2_3(\cdot,z')  
 \\
   & -\sgn(z-z')   \: \chi(D)  \sum B_i(D) (\lambda_i(D))^3 e^{-\lambda_i(D)|z -z'|} s^2_\om(\cdot,z').
\end{align*}

This time, $s^2_3 = 0$ and $\widehat{s^2_\om}$ is homogeneous of degree 1.  We get as before:
 \begin{equation} \label{I2}
 \begin{aligned}
\|  I^2_3(\cdot, z,z') \|_{L^\infty(\R^2)} & \le  \frac{C}{(1+|z-z'|)^{2}} \, \| F(\cdot, z') \|_{L^1_{uloc}(\R^2)}, \\   
\|  I^2_\omega(\cdot, z,z') \|_{L^\infty(\R^2)} & \le  \frac{C}{(1+|z-z'|)^{7/3}}  \| F(\cdot, z') \|_{L^1_{uloc}(\R^2)}, \\
\left\|  \frac{D}{|D|^2} I^2_\omega(\cdot,z,z') \right\|_{L^\infty(\R^2)} & \le  \frac{C}{(1+|z-z'|)^{2}}  \| F(\cdot, z') \|_{L^1_{uloc}(\R^2)}, \\
\left\|  \frac{D}{|D|^2} \pa_z I^2_3(\cdot,z,z') \right\|_{L^\infty(\R^2)} & \le  \frac{C}{(1+|z-z'|)^{8/3}}  \| F(\cdot, z') \|_{L^1_{uloc}(\R^2)}.
\end{aligned}
\end{equation}
Combining \eqref{I0}-\eqref{I1}-\eqref{I2}, we find 
\begin{align*}
\|v_3^\flat(\cdot,z) \|_{L^\infty(\R^2)} \:  & \le \: C \int_0^{+\infty} \frac{1}{(1+|z-z'|)^{2/3}} \frac{1}{(1+z')^{2/3}}  \ dz' \: \| (1+z^{2/3}) F\|_{L^\infty(L^1_{uloc}(\R^2))}
\\
\|�\omega^\flat(\cdot,z) \|_{L^\infty(\R^2)} \:  & \le \: C \int_0^{+\infty} \frac{1}{(1+|z-z'|)} \frac{1}{(1+z')^{2/3}}  \ dz' \:  \| (1+z^{2/3}) F\|_{L^\infty(L^1_{uloc}(\R^2))}
\end{align*}
and 
\begin{align*}
\left\| \frac{D}{|D|^2} \omega^\flat(\cdot,z)  \right\|_{L^\infty(\R^2)} \:  & \le \: C \int_0^{+\infty} \frac{1}{(1+|z-z'|)^{2/3}} \frac{1}{(1+z')^{2/3}}  \ dz' \: \:   \| (1+z)^{2/3} F\|_{L^\infty(L^1_{uloc}(\R^2)} \\
\left\| \frac{D}{|D|^2} \pa_z v_3^\flat(\cdot,z)  \right\|_{L^\infty(\R^2)} \:  & \le \: C \int_0^{+\infty} \frac{1}{(1+|z-z'|)^{4/3}} \frac{1}{(1+z')^{2/3}}  \ dz' \:  \| (1+z)^{2/3} F\|_{L^\infty(L^1_{uloc}(\R^2)} .
\end{align*}
We deduce (see Lemma \ref{lem:integrales} in the Appendix)
\begin{equation} \label{estimate_bemol_1}
\begin{aligned}
\|v_3^\flat(\cdot,z) \|_{L^\infty(\R^2)} \:  & \le \: C (1+z)^{-1/3} \:   \| (1+z^{2/3}) F\|_{L^\infty(L^1_{uloc}(\R^2))} \\
\|\omega^\flat(\cdot,z) \|_{L^\infty(\R^2)}   \:  & \le \: C (1+z)^{-2/3} \ln(2+z)  \:  \| (1+z^{2/3}) F\|_{L^\infty(L^1_{uloc}(\R^2))} 
\end{aligned}
\end{equation}
and 
\begin{equation} \label{estimate_bemol_2}
\begin{aligned}
\left\| \frac{D}{|D|^2} \omega^\flat(\cdot,z)  \right\|_{L^\infty(\R^2)} \:  & \le \: C  (1+z)^{-1/3} \:   \| (1+z^{2/3}) F\|_{L^\infty(L^1_{uloc}(\R^2))} \\
\left\| \frac{D}{|D|^2} \pa_z v_3^\flat(\cdot,z)  \right\|_{L^\infty(\R^2)} \:  & \le \: C  (1+z)^{-2/3} \ln(2+z) \:   \| (1+z^{2/3}) F\|_{L^\infty(L^1_{uloc}(\R^2))}.
\end{aligned}
\end{equation}

\bigskip
\noindent{\em High frequencies.}
  To obtain the estimate  \eqref{estimate_source}, we still have to control the high frequencies:
 \begin{equation} \label{def_vdiese}
  V^\# = \int_{\R^+} \sum_{k=0}^2 J^k(\cdot,z,z') dz', \quad   J^k(\cdot,z,z')  \defin  (1-\chi(D))  \pa^k_{z} G(D,z-z') S^k(\cdot,z').
 \end{equation} 
Instead of Lemma \ref{lemmadecay}, we shall use (see appendix for a proof)
\begin{lemma} \label{lemmadecay2}
Let $\chi \in C^\infty_c(\R^2)$, with $\chi = 1$ in a ball $B_r:=B(0,r)$ for some $r>0$. Let $P = P(\xi) \in C^3_b(\R^2\setminus B_r)$.  For $u_0 = u_0(y_1,y_2) \in H^N_{uloc}(\R^2)$, $N \in \N$, we define $u^i = u^i(y_1,y_2,z)$  by  
\begin{equation} \label{def_ui_high}
 u^i(\cdot,z)  \defin (1-\chi(D)) P(D) e^{-\lambda_i(D) z} u_0. 
 \end{equation}
Then, for $N$ large enough and $\delta > 0$ small enough: 
$$  \| e^{\delta z}  u^1 \|_{L^\infty(\R^3_+)} +  \| e^{\delta z}  u^2 \|_{L^\infty(\R^3_+)} +  \| e^{\delta z}  u^3 \|_{L^\infty(\R^3_+)}  \| \: \le \: C \|�u_0 \|_{H^N_{uloc}(\R^2)}.$$ 
\end{lemma}
\begin{remark}
As in the proof of Lemma \ref{lemmadecay}, part of the proof of Lemma \ref{lemmadecay2} gives a meaning to \eqref{def_ui_high}. In particular, it is shown that for $n$ large enough, and any $z > 0$, the  kernel 
$$ K_{n}(x_1,x_2,z) \: := \: {\cal F}^{-1} \left((1+ |\xi|^2)^{-n}  (1-\chi(\xi)) P(\xi) e^{-\lambda_i(\xi)z}  \right) $$
belongs to $L^1(\R^2)$ so that $u^i = K_n \star ( (1-\Delta)^{n} u_0)$ defines at least an element of $L^2_{uloc}$, as the convolution of functions in $L^1$ and $L^2_{uloc}$ (assuming $N \ge 2n$).  
\end{remark}
The analysis is simpler than for low frequencies. From \eqref{def_vdiese}, \eqref{def:S0}-\eqref{def:S12} and Lemma \ref{lem:G}, we decompose the  components of $J^k$ for $k=0,1,2$ into terms of the form 
$$ (1-\chi(D)) R(D)  e^{-\lambda_i(D)|z-z'|}\pa_1^{a_1} \pa_2^{a_2} F_{jl},$$
  where $F_{jl}$  are components of our source term $F$, $a_1, a_2=0,1,2$ with $1\leq a_1+a_2\leq 3$, and  $R(D)$ is of the form 
$$ R(D) = {\cal R}(\lambda_1(D), \lambda_2(D), \lambda_3(D), D),  $$
for some rational expression ${\cal R} = {\cal R}(\lambda_1,\lambda_2,\lambda_3, \xi)$. Considering the behaviour of $\lambda_i(\xi)$ at infinity (see Lemma \ref{lemmadecay} and Remark \ref{remark_roots}), it can be easily seen that 
$ |\xi|^{-2n} R(\xi) \in C^3_b(\R^2\setminus B_r) $  for some $n$ large enough. 
Thus, we can apply Lemma \ref{lemmadecay2} with 
$$P(\xi) = |\xi|^{-2n} R(\xi), \quad u_0 = (\pa_1^2 + \pa^2_2)^{n} \pa_1^{a_1}\pa_2^{a_2}  F_{jl}(\cdot,z').$$
 This shows that for $m$  large enough ($m = N+2n+3$):  
\begin{equation} \label{Jk}
\| J^k(\cdot,z,z') \|_{L^\infty(\R^2)} \: \le \: C \, e^{-\delta |z-z'|} \|F(\cdot, z') \|_{H^m_{uloc}(\R^2)}.
\end{equation} 
Also, up to taking a larger $m$, one can check that  
\begin{equation} \label{Jkbis}
\| \pa_z J^k(\cdot,z,z') \|_{L^\infty(\R^2)} \: \le \: C \, e^{-\delta |z-z'|} \|F(\cdot, z') \|_{H^m_{uloc}(\R^2)}.
\end{equation} 
We deduce from \eqref{Jk}-\eqref{Jkbis} that for $m$ large enough
\begin{equation} \label{estimate_diese1}
\begin{aligned}
&\| V^\#(\cdot,z) \|_{L^\infty(\R^2) } + \|\pa_z V^\#(\cdot,z) \|_{L^\infty(\R^2)} \\& \le C \: \int_0^{+\infty} e^{-\delta|z-z'|}(1+z')^{-2/3} \, dz' \:  \|(1+z)^{2/3} F \|_{L^\infty(H^m_{uloc})} 
\\
& \le C (1+z)^{-2/3} \: \|(1+z)^{2/3}  F \|_{L^\infty(H^m_{uloc})}. 
\end{aligned}
\end{equation}
Together  with \eqref{estimate_bemol_1}, this inequality implies the estimate \eqref{estimate_source}. Together with \eqref{estimate_bemol_2}, it further yields 
\begin{equation} \label{estimate_source2}
 \left\| (1+z)^{1/3} \frac{D}{|D|^2} \pa_z v_3 \right\|_{L^\infty(\R^3_+)} + \left\| (1+z)^{1/3}\frac{D}{|D|^2} \omega \right\|_{L^\infty(\R^3_+)} \le C \|(1+z)^{2/3} F \|_{L^\infty(H^m_{uloc})}. 
 \end{equation}

\subsection{Proof of Theorem \ref{thm_SCsource}}
In the last section, we have constructed a particular solution  of \eqref{OS1}-\eqref{OS2} satisfying \eqref{estimate_source}, \eqref{estimate_source2}; in the rest of this paragraph, we denote this particular solution as $V^p=(v^p_3,\omega^p)^t$. The bound \eqref{estimate_source2} implies in particular that 
\begin{equation}
 \| (1+z)^{1/3} (v_1^p,v_2^p) \|_{L^\infty(\R^3_+)} \: \le \: C \, \|�(1+z)^{2/3} F \|_{L^\infty(H^m_{uloc})}  
 \end{equation}
 where  $v_1^p,v_2^p$ are recovered from $v_3^p,\omega^p$ through formula \eqref{formulav1v2}. 
 
 \medskip
We still need to make the connection with the solution of \eqref{SC_source}.  Following the discussion after Theorem \ref{thm_SCsource}, for smooth and compactly supported data, such a solution exists, and the point is to establish \eqref{theestimate}. We introduce 
$$ u \defin v - v^p, \quad w  = \omega - \omega^p$$
 Functions $u_3$ and $v$ satisfy the homogeneous version of the Orr-Sommerfeld equations: 
 \begin{equation} \label{OSH}
 \pa_3 u_3 + \Delta w = 0, \quad  -\pa_3 w+ \Delta^2 u_3  = 0. 
 \end{equation}
These equations are completed by the boundary conditions 
\be\ba\label{BC}
u_3\vert_{z=0} = v_{0,3} - v_3^p\vert_{z=0}, \quad \pa_z u_3\vert_{z=0} = -\pa_1 (v_{0,1} - v^p_1) -\pa_2 (v_{0,2} - v^p_2), \\
 w\vert_{z=0} = \pa_1 (v_{0,2} - v^p_2)  - \pa_2 (v_{0,1} - v^p_1). 
\ea\ee
 System \eqref{OSH}-\eqref{BC} is the formulation in terms of vertical velocity and vorticity of  a Stokes Coriolis system with zero  source term and  inhomogeneous Dirichlet data. Formal solutions read  
 \begin{equation} \label{formula}
\left( \begin{matrix} \hat{u}_3(\xi,z)  \\ \hat{w}(\xi,z)  \end{matrix} \right) = \sum_{i=1}^3 e^{-\lambda_i(\xi) z}  C_i(\xi)  V_i^-(\xi) 
 \end{equation}
 where coefficients $C_i$ obey  the system 
 \begin{equation}
 \begin{pmatrix} 
1 & 1  & 1 \\ 
\lambda_1 & \lambda_2 & \lambda_3 \\
\Omega_1 & \Omega_2  & \Omega_3 
\end{pmatrix} 
\begin{pmatrix} C_1 \\ C_2 \\ C_3  \end{pmatrix} =  \begin{pmatrix} \hat{u}_3\vert_{z=0} \\ -\pa_z \hat{u}_3\vert_{z=0}  \\-\hat w\vert_{z=0} \end{pmatrix} .
 \end{equation}
The determinant $D_3$ of this system reads 
$$ \label{D} 
D_3 \defin (\lambda_2-\lambda_1)(\Om_3-\Om_1) - (\lambda_3-\lambda_1) (\Om_2-\Om_1),
$$
so that $D_3\to \bar D_3\in \C^*$ as $\xi\to 0*$.

After  tedious computation, we find
\begin{equation} \label{C_i}
\begin{aligned} 
C_1 & = \frac{1}{D_3} \left(  (\lambda_2 \Omega_3 - \lambda_3 \Omega_2) \hat{u}_3\vert_{z=0} + (\Om_3 - \Om_2)\pa_z  \hat{u}_3\vert_{z=0}  
 +  (\lambda_2-\lambda_3)  \hat{w}\vert_{z=0} \right), \\
C_2 & = \frac{1}{D_3} \left((\lambda_3 \Om_1 - \lambda_1 \Om_3) \hat{u}_3\vert_{z=0}+ (\Omega_1  - \Omega_3) \pa_z  \hat{u}_3\vert_{z=0}    + (\lambda_3-\lambda_1)   \hat{w}\vert_{z=0}  \right) ,\\
C_3 & = \frac{1}{D_3} \left((\lambda_1 \Om_2 - \lambda_2 \Om_1) \hat{u}_3\vert_{z=0} + (\Om_2-\Om_1) \pa_z  \hat{u}_3\vert_{z=0}  + (\lambda_1 - \lambda_2)   \hat{w}\vert_{z=0}  \right).
\end{aligned}
\end{equation}
Nevertheless, the expressions in \eqref{formula} are not necessarily well defined, due to possible singularities at $\xi = 0$. In particular, if we want to apply Lemma \ref{lemmadecay}, we need the coefficient in front of $e^{-\lambda_1(\xi) z}$ to contain somehow some positive power of  $\xi$. 
Using the asymptotics of Lemma \ref{asymptote_lambda}, we compute
\begin{align}
 |C_1(\xi)| & \le | \hat{u}_3\vert_{z=0} | \: + \:  |\pa_z  \hat{u}_3\vert_{z=0}| +   |\hat{w}\vert_{z=0} | , \\
 |C_2(\xi)| & \le |\xi|  \, | \hat{u}_3\vert_{z=0} | \: + \:  |\pa_z  \hat{u}_3\vert_{z=0}|  + |\hat{w}\vert_{z=0}|, \\
 |C_3(\xi)| & \le |\xi| \, |\hat{u}_3\vert_{z=0} |  + |\pa_z  \hat{u}_3\vert_{z=0}|  +  |\hat{w}\vert_{z=0}| .
 \end{align}
 for small $|\xi|$. The asymptotics is given by 
 \begin{lemma} \label{boundarydatalowf}

\begin{enumerate}
\item 
The boundary data $\pa_z \hat u_3\vert_{z=0} , \hat w\vert_{z=0}$ in \eqref{BC}, as well as $\hat v_{0,3}\vert_{z=0}$ (which appears in 
$\hat{u}_3\vert_{z=0}$)   ``contain a power of $\xi$ at low frequencies". More precisely, for $\xi$ small enough, they can all be decomposed into terms of the form $\xi\cdot \hat f$, for some $f\in L^2_{uloc}(\R^2)$.   As a consequence, for any function $Q\in \mathcal C^\infty(\R^2)$,
\begin{multline*}
\left\| \chi(D) Q(D)\exp(-\lambda_1(D)z) \begin{pmatrix}
\pa_z \hat u_3\vert_{z=0}\\ \hat w\vert_{z=0}\\ \hat v_{0,3}\vert_{z=0} 
\end{pmatrix}\right\|_{L^\infty(\R^2)} \\\leq C (1+z)^{-1/3}\left(  \| (v_{0,1}, v_{0,2}) \|_{L^2_{uloc}(\R^2)} +  \| (\nu_1, \nu_2) \|_{L^2_{uloc}(\R^2)} + � \|(1+z)^{2/3} F \|_{H^m_{uloc}(\R^3_+)}\right)
\end{multline*}
and for $j=2,3$,
\begin{multline*}
\left\| \chi(D) Q(D)\exp(-\lambda_j(D)z) \begin{pmatrix}
\pa_z \hat u_3\vert_{z=0}\\ \hat w\vert_{z=0}\\ \hat v_{0,3}\vert_{z=0} 
\end{pmatrix}\right\|_{L^\infty(\R^2)} \\\leq C e^{-\delta z}\left(  \| (v_{0,1}, v_{0,2}) \|_{L^2_{uloc}(\R^2)} +  \| (\nu_1, \nu_2) \|_{L^2_{uloc}(\R^2)} + � \|(1+z)^{2/3} F \|_{H^m_{uloc}(\R^3_+)}\right).
\end{multline*}

\item
Concerning the boundary data $u_3^p\vert_{z=0}$ (which is the other term in $\hat{u}_3\vert_{z=0}$), we have, for any function $Q\in \mathcal C^\infty(\R^2)$,
$$\ba
\left\|\left( \chi(D) Q(D) \exp(-\lambda_1(D) z)\right) u_3^p\vert_{z=0} \right\|_{L^\infty(\R^2)}\leq C (1+z)^{-1/3}\|F\|_{L^1_{uloc}(\R^2)},\\
\left\|\left( \chi(D) Q(D) \exp(-\lambda_j(D) z)\right) u_3^p\vert_{z=0} \right\|_{L^\infty(\R^2)}\leq C e^{-\delta z}\|F\|_{L^1_{uloc}(\R^2)}.
\ea
$$
\end{enumerate}
 \end{lemma}
\begin{proof}   The first part of the statement is obvious for the  last two boundary data, namely 
 $$\pa_z u_3\vert_{z=0} = -\pa_1 (v_{0,1} - v^p_1) -\pa_2 (v_{0,2} - v^p_2), \quad \mbox{ and } \:  w\vert_{z=0} = \pa_1 (v_{0,2} - v^p_2)  - \pa_2 (v_{0,1} - v^p_1). $$
 It remains to consider $ v_{0,3} $. This is where the assumption on $v_{0,3}$ in the theorem plays a role. Indeed, we have $v_{0,3} = - \pa_1 \nu_1 - \pa_2 \nu_2$, so that it satisfies the properties of the lemma.   The estimate is then a straightforward consequence of Lemma \ref{lemmadecay}.

\medskip 
The former argument does not work with the boundary data $u_3^p\vert_{z=0}$: indeed, if we factor out crudely a power of $\xi$ from the integral defining it, then the convergence of the remaining integral is no longer clear. Therefore we go back to the definition of $u^p_3$: we have, using the notations of \eqref{def_vbemol},
 $$
\chi(D) u_3^p\vert_{z=0}= \int_{\R_+}\sum_{k=0}^2 I^k_3(\cdot, 0, z')dz'.
 $$
 It can be easily checked that the terms with $I^k_3$ for $k=1,2$ do not raise any difficulty (in fact,  the trace stemming from these two terms contains a power of $\xi$ at low frequencies.) Thus we focus on 
 \begin{eqnarray*}
&& \int_{\R_+}I^0_3(\cdot, 0, z')dz'\\
&=& \int_{\R_+}  \left(\chi(D)   \sum_{i=1}^3 A_i(D) e^{-\lambda_i(D)z'}   s^0_3(\cdot,z')  
 \: + \: \chi(D)  \sum_{i=1}^3 B_i(D) e^{-\lambda_i(D)z'} s^0_\om(\cdot,z')\right)\:dz'. 
 \end{eqnarray*}
Applying $\exp(-\lambda_j (D)z)$ we have to estimate the $L^\infty(\R^2)$ norms of
$$\ba
 \int_{\R_+} \chi(D) Q(D)  \sum_{i=1}^3 A_i(D) e^{-\lambda_i(D)z' - \lambda_j(D)z}   s^0_3(\cdot,z')  \:dz'\\
 \text{and }\int_{\R_+}  \: \chi(D) Q(D) \sum_{i=1}^3 B_i(D) e^{-\lambda_i(D)z' - \lambda_j(D)z}  s^0_\om(\cdot,z')\:dz'.
 \ea
$$
We recall that $\widehat{s^0_3}(\xi,z')$, resp. $\widehat{s^0_\om}(\xi,z')$ is a product of components of $\hat{F}(\xi,z')$ by  homogeneous polynomials of  degree 2, resp. degree 3 in $\xi$, and that the behaviour of $A_i, B_i$ is given in Lemma~\ref{lem:G}. When $i=j=1$, using Lemma \ref{lemmadecay} and Lemma \ref{lem:integrales} in the Appendix, the corresponding integral is bounded by
\begin{multline*}
\int_{\R_+} \frac{1}{(1+ z + z')^{2/3}} \frac{1}{(1+z')^{2/3}} \:dz'\ \| (1+z)^{2/3} F \|_{L^\infty(L^1_{uloc})}\\
\leq C (1+z)^{-1/3}\| (1+z)^{2/3} F \|_{L^\infty(L^1_{uloc})}.
\end{multline*}
When $i=2,3$, the integral is bounded by
\begin{multline*}
\int_{\R_+} \frac{\exp(-\delta z')}{(1+ z )^{2/3}} \frac{1}{(1+z')^{2/3}} \:dz'\ \| (1+z)^{2/3} F \|_{L^\infty(L^1_{uloc})}\\
\leq C (1+z)^{-2/3}\| (1+z)^{2/3} F \|_{L^\infty(L^1_{uloc})}.
\end{multline*}
When $j=2,3$,  the integral is bounded by
\begin{multline*}
\int_{\R_+} \frac{\exp(-\delta z)}{(1+ z' )^{2/3}} \frac{1}{(1+z')^{2/3}} \:dz'\ \| (1+z)^{2/3} F \|_{L^\infty(L^1_{uloc})}\\
\leq C\exp(-\delta z)\| (1+z)^{2/3} F \|_{L^\infty(L^1_{uloc})}.
\end{multline*}
Gathering all the terms we obtain the estimate announced in the Lemma.
\end{proof} 
 
Going back to \eqref{formula}, we infer that
\begin{eqnarray}
&&(1+z)^{1/3}\| \chi(D) u_3(\cdot, z)\|_{L^\infty(\R^2)}+(1+z)^{2/3}\| \chi(D) w(\cdot, z)\|_{L^\infty(\R^2)}\label{est:sol-totale}\\
&\leq & C \left(  \| (v_{0,1}, v_{0,2}) \|_{L^2_{uloc}(\R^2)} +  \| (\nu_1, \nu_2) \|_{L^2_{uloc}(\R^2)} +  \|(1+z)^{2/3} F \|_{H^m_{uloc}(\R^3_+)}\right),\nonumber
\end{eqnarray}
Then, for further control of the horizontal components $(v_1,v_2)$, one would  like an analogue of \eqref{estimate_bemol_2}, namely a bound like: 
 \begin{eqnarray*}
 &&(1+z)^{1/3}\left\| \frac{D}{|D|^2} \chi(D)\pa_z u_3(\cdot, z)\right\|_{L^\infty(\R^2)}+(1+z)^{1/3}\left\|  \frac{D}{|D|^2} \chi(D) w(\cdot, z)\right\|_{L^\infty(\R^2)}\\
 &\leq & C \left(  \| (v_{0,1}, v_{0,2}) \|_{L^2_{uloc}(\R^2)} +  \| (\nu_1, \nu_2) \|_{L^2_{uloc}(\R^2)} +  \|(1+z)^{2/3} F \|_{H^m_{uloc}(\R^3_+)}\right).
 \end{eqnarray*}
However, such an estimate is not clear. Indeed, in view of \eqref{formula}, we have: 
\begin{equation*}
 \chi(D)  \begin{pmatrix} \pa_z u_3(\cdot, z) \\ w(\cdot, z) \end{pmatrix} = \chi(D) \sum_{i=1}^3 e^{-\lambda_i(D) z}  \begin{pmatrix} -\lambda_i(D) C_i   \\ -\Omega_i(D) C_i \end{pmatrix}. 
 \end{equation*}
  The term with index $i=1$ does not raise any difficulty, because $\lambda_1(D)$ and $\Omega_1(D)$ bring  extra powers of $\xi$, which are enough to apply Lemma \ref{lemmadecay}. But the difficulty comes from indices $2$ and $3$. For instance, they involve terms of the type 
$$ \chi(D) P_0(D) e^{-\lambda_{2,3}(D)} \hat{v}_{0}, \quad \mbox{with} \:  P_0 \:  \mbox{homogenenous of degree} \:  0, $$
and therefore are not covered by Lemma \ref{lemmadecay} : with the notations of the lemma, one has $\alpha = 0$, which is not enough.  Typically, these  homogeneous functions of degree zero involve Riesz transforms, meaning $P_0(\xi) = \frac{\xi_k \xi_l}{|\xi|^2}$, $k,l=1,2$. 

Hence, one must use extra cancellations. We recall that in view of  \eqref{formulav1v2}, we want to exhibit cancellations in $|D|^{-2}(D_1\pa_z u_3 + D_2 w)$ and in $|D|^{-2}(D_2\pa_z u_3 - D_1 w)$. Let us comment briefly on the first term.  We compute $(-\xi_1\lambda_i - \xi_2 \Om_i) C_i$ for $i=2,3$ in terms of the boundary data. Setting $u_0=v_0-v^p\vert_{z=0}$, we find that
\begin{eqnarray*}
C_2(\xi)&=& \frac{1}{D_3} (\lambda_3 \Om_1 - \lambda_1 \Om_3) \hat u_{0,3}  \\
&+& \frac{1}{D_3}\left[ \left((\Om_3-\Om_1) i\xi_1 - i\xi_2 (\lambda_3-\lambda_1)\right) \hat u_{0,1} + \left((\Om_3-\Om_1) i\xi_2 + i\xi_1 (\lambda_3-\lambda_1)\right) \hat u_{0,2}\right].
\end{eqnarray*}
 We then use the asymptotic formulas of Lemma \ref{asymptote_lambda}. In particular,
 $$\ba
 (-\xi_1\lambda_2 - \xi_2 \Om_2)\left((\Om_3-\Om_1) i\xi_1 - i\xi_2 (\lambda_3-\lambda_1)\right)=|\xi|^2 + O(|\xi|^3),\\
  (-\xi_1\lambda_2 - \xi_2 \Om_2)\left((\Om_3-\Om_1) i\xi_2 + i\xi_1 (\lambda_3-\lambda_1)\right) =-i|\xi|^2 + O(|\xi|)^3.
 \ea
 $$
A similar formula holds for $C_3$. It follows that there exist $Q_2,Q_3\in \mathcal C^\infty(\R^2)^2$ such that
\begin{eqnarray*}
&&\mathcal F(\chi(D)|D|^{-2}(D_1\pa_z u_3 + D_2 w))\\&=&\chi(\xi)\frac{-\xi_1 \lambda_1 - \xi_2 \Om_1}{D_3|\xi|^2} e^{-\lambda_1(\xi )z}C_1(\xi)
\\
&+& \frac{1}{D_3} \left[(\lambda_3 \Om_1 - \lambda_1 \Om_3)(-\xi_1\lambda_2 - \xi_2 \Om_2)e^{-\lambda_2 z} + (\lambda_1 \Om_2 - \lambda_2 \Om_1)(-\xi_1\lambda_3 - \xi_2 \Om_3)e^{-\lambda_3 z} \right]\hat u_{0,3}\\
&+& \sum_{i=2,3}\chi(\xi ) e^{-\lambda_i z}Q_i(\xi) \cdot \hat u_{0,h}(\xi,z).
\end{eqnarray*}
The first two terms are treated in the same way as Lemma \ref{boundarydatalowf}, factoring out a power of $\xi$ when necessary, and going back to the definition of $v^p$. We leave the details to the reader. The inverse Fourier transform of the last term is $\mathcal F^{-1} (\chi Q_i e^{-\lambda_i z})\ast u_{0,h}$, which is bounded in $L^\infty(\R^2)$ by $e^{-\delta z}\|u_{0,h}\|_{L^2_{uloc}}$. Similar statements hold for $\chi(D)|D|^{-2}(-\pa_z D_2 u_3 + D_1 w )$.
  It follows that
\begin{multline}\label{est:u-BF}
(1+z)^{1/3} \|\chi(D) u(\cdot, z)\|_{L^\infty(\R^2)}\\
 \leq C \left(  \| (v_{0,1}, v_{0,2}) \|_{L^2_{uloc}(\R^2)} +  \| (\nu_1, \nu_2) \|_{L^2_{uloc}(\R^2)} +  \|(1+z)^{2/3} F \|_{H^m_{uloc}(\R^3_+)}\right).
\end{multline}

  We now address the estimates of $\hat u(\xi, z)$ for large frequencies. The arguments are very close to the ones developed after Lemma \ref{lemmadecay2}. Using \eqref{formula} and \eqref{C_i}, we find that for $|\xi|\gg 1$, $\hat u_3(\xi, z)$ and $\hat w(\xi, z)$ can be written as linear combinations of terms of the type
  $$
  R_{ij}(\lambda_1, \lambda_2, \lambda_3, \xi) \exp(-\lambda_i(\xi) z) \hat g_j(\xi),\quad 1\leq i,j\leq 3,
  $$
  where $g_1=u_3\vert_{z=0}$, $g_2=\pa_z u_3\vert_{z=0}$ and $g_3= w\vert_{z=0}$ and $R_{ij}$ is a rational expression. Thus, using Lemma \ref{asymptote_lambda} and Lemma \ref{lem:G}, there exists $n\in \N$ such that $|\xi|^{-2n}   R_{ij}(\lambda_1, \lambda_2, \lambda_3, \xi) $ is bounded as $|\xi|\to\infty$ for all $i,j$. Lemma \ref{lemmadecay2} then entails that for some $N$ sufficiently large,
  $$
  \ba
 \| (1-\chi)(D) u_3(\cdot, z)\|_{L^\infty(\R^2)}\leq C e^{-\delta z}\sum_{j=1}^3\|g_j\|_{H^N_{uloc}},\\
 \| (1-\chi)(D) w(\cdot, z)\|_{L^\infty(\R^2)}\leq C  e^{-\delta z}\sum_{j=1}^3\|g_j\|_{H^N_{uloc}},\ea
$$
  and similar estimates hold for $ \frac{D}{|D|^2}\pa_z u_3$ and $ \frac{D}{|D|^2} w$. Using \eqref{BC} and \eqref{Jk}-\eqref{Jkbis}, we infer that for some $m\geq 1$ large enough,
\be\label{est:u-HF}
  \|(1-\chi(D)) u(\cdot, z)\|_{L^\infty(\R^2)}\leq Ce^{-\delta z} \left(  \| v_0 \|_{H^{m+\frac{1}{2}}_{uloc}(\R^2)} +  \|(1+z)^{2/3} F \|_{H^m_{uloc}(\R^3_+)}\right).
\ee
  Gathering \eqref{est:u-BF} and \eqref{est:u-HF}, we deduce that $u$ satisfies the estimate:
  $$  \|(1+z)^{1/3} u \|_{L^\infty} \le C \: \left(   \| v_0 \|_{H^{m+\frac{1}{2}}_{uloc}(\R^2)} +  +  \| (\nu_1, \nu_2) \|_{L^2_{uloc}(\R^2)}  + \|(1+z)^{2/3} F \|_{H^m_{uloc}(\R^3_+)}\right) $$
for $m$ large enough. Thus, in view of the estimate \eqref{estimate_source} satisfied by $v^p$, $v=u+ v^p$ is a solution of \eqref{SC_source} satisfying 
  $$  \|(1+z)^{1/3} v \|_{L^\infty} \le C \: \left(   \| v_0 \|_{H^{m+\frac{1}{2}}_{uloc}(\R^2)} +  +  \| (\nu_1, \nu_2) \|_{L^2_{uloc}(\R^2)}  + \|(1+z)^{2/3} F \|_{H^m_{uloc}(\R^3_+)}\right) $$
  for $m$ large enough. It remains to go to the higher regularity bound \eqref{theestimate}. First,  up to taking a slightly larger $m$, we clearly have: 
$$  \|(1+z)^{1/3} \na  v \|_{L^\infty} \le C \: \left(   \| v_0 \|_{H^{m+\frac{1}{2}}_{uloc}(\R^2)} +  +  \| (\nu_1, \nu_2) \|_{L^2_{uloc}(\R^2)}  + \|(1+z)^{2/3} F \|_{H^m_{uloc}(\R^3_+)}\right). $$
This follows from direct differentiation of formula \eqref{integralV} satisfied by $v^p$ and formula \eqref{formula} satisfied by $u = v - v^p$. Clearly, the differentiation is harmless, in particular at low frequencies where it may even add positive powers of $\xi$. 
It follows that our solution belongs to $H^1_{uloc}(\R^3_+)$, and thus enters the framework of local elliptic regularity theory for the Stokes equation. 
In particular, for any $k \in \Z^3$ with $k_z \le 2$:  
\begin{align*}
 & \| v \|_{H^{m+1}(B(k,1) \cap \Omega_{bl})} \\ 
 & \le C 
\left( \| v_0 \|_{H^{m+\frac{1}{2}}_{uloc}(\R^2)}   +  \| (\nu_1, \nu_2) \|_{L^2_{uloc}(\R^2)}  + \| F \|_{H^m_{uloc}(\R^3_+)}  +  \| v \|_{H^1(B(k,2)  \cap \Omega_{bl})}    \right) \\
 & \le C 
\left( \| v_0 \|_{H^{m+\frac{1}{2}}_{uloc}(\R^2)}   +  \| (\nu_1, \nu_2) \|_{L^2_{uloc}(\R^2)}  + \| F \|_{H^m_{uloc}(\R^3_+)}  +  \| v \|_{H^1_{uloc}(\R^3_+)}    \right)
 \end{align*} 
  and for any $k \in \Z^3$ with $k_z > 2$: 
  \begin{align*}
 & \| v \|_{H^{m+1}(B(k,1) \cap \Omega_{bl})} \\ 
 & \le C    
\left( \| F \|_{H^m(B(k,2) \cap \Omega_{bl})}  +  \| v \|_{H^1(B(k,2)  \cap \Omega_{bl})}    \right) \\
& \le C |k_z|^{-1/3} \left( \| (1+z)^{2/3} F \|_{H^m_{uloc}(\R^3_+)}  +  \|(1+y)^{1/3} v \|_{H^1_{uloc}(\R^3_+)}     \right)
 \end{align*} 
The bound \eqref{theestimate} follows.

\section{Proof of Theorem \ref{Thm_WP}}

\subsection{Navier-Stokes Coriolis system in a half space}

This paragraph is devoted to the well-posedness of system \eqref{NSC1} under a smallness assumption. Once again, we can assume $M=0$ with no loss of generality. Following the analysis of the linear case performed in the previous section, we introduce the functional spaces
$$
\cH^m:=\left\{ v \in H^{m}_{loc}(\R^3_+),\ \|(1+y_3)^{1/3} v\|_{H^m_{uloc}} <+\infty\right\},\quad m\geq 0,
$$
and we set $\|v\|_{\cH^m} =C_m\|(1+y_3)^{1/3} v\|_{H^m_{uloc}} $, where the constant $C_m$ is chosen so that if $u,v\in (\mathcal H^m)^3$ for some $m>3/2$, then
$$
\|u\otimes v\|_{\cH^m}\leq \|u\|_{\cH^m}\|v\|_{\cH^m}.
$$
 Clearly $\cH^m$ is a Banach space for all $m\geq 0$.

The result proved in this section is the following :
\begin{proposition}
\label{prop:NS-1}
Let $m\in \N, m\gg 1$. There exists $\delta_0>0$ such that for all $v_0\in H^{m+1}_{uloc}(\R^2)$ such that $v_{0,3} = \pa_1 \nu_1 + \pa_2 \nu_2$, with $\nu_1$, $\nu_2$ in 
$L^2_{uloc}(\R^2)$ and
\be\label{hyp:smallness}
\| v_0 \|_{H^{m+1}_{uloc}(\R^2)} +  \| (\nu_1, \nu_2) \|_{L^2_{uloc}(\R^2)} \leq \delta_0,\ee
then the system
$$
\left\{
\begin{aligned}
v \cdot \na v + e \times v  + \na p - \Delta v  & = 0 \inside \{y_3 > 0\} \\
\div v & = 0  \inside \{y_3 > 0\} \\
v\vert_{ \{y_3 = 0\} } & = v_0
\end{aligned}
\right.
$$
has a unique solution in $\cH^{m+1}$.
\end{proposition}
\begin{remark}
The integer $m$ for which this result holds is the same as the one in Theorem \ref{thm_SCsource}.
\end{remark}
\begin{proof}
Proposition \ref{prop:NS-1} is an easy consequence of the fixed point theorem in $\cH^{m+1}$. For any  $v_0\in H^{m+1}_{uloc}(\R^2)$ such that $v_{0,3} = \pa_1 \nu_1 + \pa_2 \nu_2$, with $\nu_1$, $\nu_2$ in 
$L^2_{uloc}(\R^2)$, we introduce the mapping $T_{v_0}:\cH^{m+1}\to \cH^{m+1}$ such that $T_{v_0}(u)=v$ is the solution of \eqref{SC_source} with $F=u\otimes u$. Notice that $\|(1+z)^{2/3} F\|_{H^m_{uloc}}\leq \|u\|_{\cH^m}^2$. As a consequence, according to Theorem \ref{thm_SCsource}, there exists a constant $C_0$ such that for all $u\in \cH^{m+1}$,
\be\label{estT}
\|T_{v_0}(u)\|_{\cH^{m+1}}\leq C_0\left( \| v_0 \|_{H^{m+1}_{uloc}(\R^2)} +  \| (\nu_1, \nu_2) \|_{L^2_{uloc}(\R^2)} + \|u\|_{\cH^{m+1}}^2\right).
\ee

Let $\delta_0<\frac{1}{4C_0^2}$, and assume that \eqref{hyp:smallness} is satisfied. Thanks to the assumption on $\delta_0$, there exists $R_0>0$ such that
\be\label{R_0}
C_0(\delta_0 + R_0^2)\leq R_0.
\ee
Moreover, $R_0\in [R_-, R_+]$, where
$$
R_\pm=\frac{1}{2C_0}(1\pm\sqrt{1-4\delta_0 C_0^2}).
$$
Therefore $0<R_-<(2C_0)^{-1}$, and we can always choose $R_0$ so that $2R_0 C_0<1$.
Then according to \eqref{hyp:smallness}, \eqref{estT} and \eqref{R_0},
$$
\|u\|_{\cH^{m+1}}\leq R_0\Rightarrow \|T_{v_0}(u)\|_{\cH^{m+1}}\leq R_0.
$$
Moreover, if $\|u^1\|_{\cH^{m+1}}, \|u^2\|_{\cH^{m+1}}\leq R_0$, then setting $w=T_{v_0}(u^1)-T_{v_0}(u^2)$, $w$ is a solution of \eqref{SC_source} with $w_{|z=0}=0$ and with a source term $F^1-F^2= u^1\otimes u^1 - u^2\otimes u^2$. Thus, using once again Theorem \ref{thm_SCsource} and the normalization of $\|\cdot\|_{\cH^m}$,
$$
\|T_{v_0}(u^1)-T_{v_0}(u^2)\|_{\cH^{m+1}}\leq C_0 \|F^1-F^2\|_{\cH^m}\leq 2 C_0 R_0 \|u^1-u^2\|_{\cH^{m+1}}. 
$$
Notice that in the inequality above, we have assumed that $\|\cdot \|_{\cH^m}\leq \|\cdot \|_{\cH^{m+1}}$, which is always possible if the normalization constant $C_m$ is chosen sufficiently small (depending on $C_{m+1}$, $m$ being large but fixed).

Thus, since $2 C_0 R_0 <1$,  $T_{v_0}$ is a contraction over the ball of radius $R_0$ in $\cH^{m+1}$. Using Banach's fixed point theorem, we infer that $T_{v_0}$ has a fixed point in $\cH^{m+1}$. This concludes the proof of Proposition \ref{prop:NS-1}.
\end{proof}

\subsection{Navier-Stokes-Coriolis system over a bumped half-plane}

We now address the study of the full system \eqref{BL}. We follow the steps outlined in the introduction which we recall here for the reader's convenience: we first prove that there exists a solution $(v^-, p^-)$ of the system \eqref{NSC2} for $\phi,\psi$ in some function spaces to be specified, then construct the solution $(v^+, p^+)$ of \eqref{NSC1} with $v^+\vert_{|y_3=M}= v^-\vert_{|y_3=M}$. Eventually, we define a mapping $\cF$ by $\cF(\phi,\psi):=\Sigma(v^+,p^+)e_3\vert_{y_3=M} -\psi$. We recall that $v=\mathbf 1_{y_3\geq M} v^+ + \mathbf 1_{y_3<M} v^-$ is a solution of \eqref{BL} if and only if $\cF(\phi,\psi)=0$. The goal is therefore to show that for all $\phi$ small enough (in a function space to be specified) the equation $\cF(\phi,\psi)=0$ has a unique solution.

\medskip
{\em Step 1}. We study the system \eqref{NSC2}. We introduce the following function space for the bottom Dirichlet data:
\begin{equation} \label{calV}
\cV\defin \{\phi=(\phi_h, \phi_3),\ \phi_h \in H^2_{uloc}(\pa \Omega_{bl}),\ \phi_3\in H^1_{uloc}(\pa \Omega_{bl}), \quad \phi \cdot n\vert_{\pa \Omega_{bl}} = 0 \}
\end{equation}
and we set
$$
\|\phi\|_{\cV }\defin \|\phi_h\|_{H^2_{uloc}} + \|\phi_3\|_{H^1_{uloc}} 
$$
As for the stress tensor at $y_3=M$, since we will need to construct solutions in $H^{m+1}_{uloc}$ (see Proposition \ref{prop:NS-1}), we look for $\psi$ in the space $H^{m-\frac{1}{2}}_{uloc}(\R^2)$.
We then claim that the following result holds:
\begin{lemma} \label{WPmoins}
Let $m\geq 1$ be arbitrary.
There exists $\delta>0$ such that for all $\phi \in \cV$  and all $\psi\in H^{m-\frac{1}{2}}_{uloc}(\R^2)$ with $\|\phi\|_{\cV} \leq \delta$, $\|\psi\|_{H^{m-\frac{1}{2}}_{uloc}(\R^2)}\leq\delta$, system \eqref{NSC2} has a unique solution 
$$(v^-,p^-)\in H^1_{uloc}(\Om_{bl}^M) \times  L^2_{uloc}(\Om_{bl}^M).$$
Moreover, it satisfies  the following properties:
\begin{itemize}
\item
$H^{m+1}_{uloc}$ regularity: for all $M'\in ]\sup \gamma, M[$, 
$$(v^-, p^-)\in H^{m+1}_{uloc}(\R^2\times (M', M)) \times H^m_{uloc}(\R^2\times (M', M)), $$
$$
\mbox{with} \quad \|v^-\|_{ H^{m+1}_{uloc}(\R^2\times (M', M))}+ \|p^-\|_{H^m_{uloc}(\R^2\times (M', M))} \leq C_{M'} (\|\phi\|_{\cV} +\|\psi\|_{H^{m-\frac{1}{2}}_{uloc}(\R^2)}).
$$
\item Compatibility condition: there exists $\nu_1, \nu_2\in H^{1/2}_{uloc}$ such that $v_3^-\vert_{|y_3=M}=\na_h\cdot \nu_h$.
\end{itemize}
\label{lem:v-}
\end{lemma}

\begin{proof}  We start with an $H^1_{uloc}$ {\it a priori} estimate. We follow the computations of \cite{DalibardPrange}, dedicated to the linear Stokes-Coriolis system.  We first lift the boundary condition on $\pa \Om_{bl}$, introducing 
$$
v^L_h\defin \phi_h,\ v^L_3\defin \phi_3 - \na_h\cdot\phi_h(y_3-\gamma(y_h)).
$$
Then $\tv\defin v^--v^L$, $\tilde p = p^-$ satisfy
\be\label{NSC-h}\ba
- \Delta \tv + (v^L + \tv)\cdot \na \tv + \tv \cdot \na v^L +   e_3 \wedge \tv + \na \tilde p = f \text{ in } \Om_{bl}^M,\\
\dv \tv=0 \text{ in } \Om_{bl}^M,\\
\tv\vert_{\pa \Om_{bl}}=0,\\
\left( \pa_3 \tv - \left(\tilde p + \frac{|\tv + v^L|^2 }{2}\right) e_3\right)\vert_{y_3=M}=\psi -\pa_3 v^L\vert_{y_3=M}\defin \tilde \psi,\ea
\ee
where $f=-\Delta v^L + v^L\cdot \na v^L + e_3 \wedge v^L$.

\medskip
Notice that thanks to the regularity assumptions on $\phi$ and $\nu$, we have $\tilde \psi\in L^2_{uloc}(\R^2)$ and $f\in H^{-1}_{uloc}(\R^2)$. We then perform energy estimates on the system \eqref{NSC-h}, following the strategy of G\'erard-Varet and Masmoudi in \cite{DGVNMnoslip}, which is inspired from the work of Ladyzhenskaya and Solonnikov \cite{LS}. The idea is to work with the truncated energies
\begin{equation}\label{def:E_k}
E_k:=\int_{\Omega^M_{bl} \cap \{ (y_1,y_2) \in  [-k,k]^2\}} \nabla \tv\cdot\nabla \tv,
\end{equation}
and to derive an induction inequality on $(E_k)_{k\in \N}$. To that end, we consider a truncation function $\chi_k\in \mathcal C^\infty_0(\R^2)$ such that $\chi_k\equiv 1 $ in $[-k,k]^2$, $\Supp \chi_k \subset [-k-1, k+1]^2$, and $\chi_k, \chi_k',\chi_k''$ are bounded uniformly in $k$. Along the lines of  \cite{DalibardPrange}, we multiply \eqref{NSC-h} by the test function
\begin{eqnarray*}
\varphi&=&
\left(
\begin{array}{c}
\varphi_h\\
\nabla\cdot\Phi_h
\end{array}
\right):=\left(\begin{array}{c}
\chi_k \tv_h\\
-\nabla_h\cdot\left(\chi_k\int_{\gamma(y_h)}^{y_3}\tv_h(y_h,z)dz\right)
\end{array}
\right)\in H^1(\Omega^b),\\
&=&\chi_k \tv - \begin{pmatrix}
0\\ \na_h \chi_k(y_h)\cdot \int_{\gamma(y_h)}^{y_3}\tv_h(y_h,z)dz
\end{pmatrix}.
\end{eqnarray*}
Since this test function is divergence-free, there is no commutator term stemming from the pressure. In the work \cite{DalibardPrange}, an inequality of the following type is derived:
 $$  E_k \leq C \left( (E_{k+1} - E_k) + (\|\phi\|_{\cV}^2 + \|\psi\|^2_{H^{-\frac{1}{2}}_{uloc}})(k+1)^2\right).  $$
This discrete differential inequality is a key {\it a priori estimate}, that allows for the construction of a solution.  Indeed, introducing an approximate solution $\tilde v^n$ for $|y_1,y_2| \le n$, say with Dirichlet boundary conditions at the lateral boundary, a standard estimate yields that 
$ E_n   \le C n$, where this time $E_k = \int |\chi_k \na \tilde v^n|^2$.  Combining this information with above induction relation allows to obtain a uniform bound 
on the $E_k$'s of the type $E_k\leq C k^2$, from which we deduce  a $H^1_{uloc}$ bound on $\tilde v^n$ uniformly in $n$. From there, one obtains an exact solution by compactness. We refer to \cite{DalibardPrange} for more details. 

\medskip
Here, there are two noticeable differences with \cite{DalibardPrange}:
\begin{itemize}
\item The boundary condition at $y_3=M$ in \eqref{NSC-h} does not involve a Dirichlet-Neumann operator, which makes things easier.
\item On the other hand, one has to handle the quadratic terms $(v^L + \tv)\cdot \na \tv + \tv \cdot \na v^L$, which explains the introduction of the $|v|^2$ in the stress tensor at $y_3=M$.
\end{itemize}
Therefore we focus on the treatment of these nonlinear terms. The easiest one is
$$
\left|\int_{\Om_{bl}^M}\left(\tv \cdot \na v^L \right)\cdot \varphi \right| \leq C \|\phi\|_{\cV}  E_{k+1},
$$
where the constant $C$ depends only on $M$ and on $\|\gamma\|_{W^{1,\infty}}$.
On the other hand,
\begin{eqnarray*}
\int_{\Om_{bl}^M} \left((v^L + \tv)\cdot \na \tv\right)\cdot (\chi_k \tv)&=& \int_{\Om_{bl}^M} \chi_k(v^L + \tv)\cdot \na \frac{|\tv|^2}{2}\\
&=&- \int_{\Om_{bl}^M} \frac{|\tv|^2}{2}(v^L + \tv) \cdot \na \chi_k+\int_{\R^2} \chi_k \left((v^L_3 + \tv_3) \frac{|\tv|^2}{2}\right)\vert_{y_3=M}.
\end{eqnarray*}
The first term in the right-hand side is bounded by $C(E_{k+1}-E_k)^{3/2} + C\|\phi\|_{\cV} (E_{k+1}-E_k)$. We group the second one with the boundary terms stemming from the pressure and the laplacian. The sum of these three boundary terms is
$$
\int_{\R^2} \chi_k \left[ -\pa_3 \tv \cdot \tv +  (v^L_3 + \tv_3) \frac{|\tv|^2}{2} + p^- \tv_3 \right]\vert_{y_3=M}.
$$
Using the boundary condition in \eqref{NSC-h}, the integral above is equal to
$$
-\int_{\R^2} \chi_k \tv\vert_{y_3=M}\cdot \left(\tilde \psi + \left(\tilde v \cdot v^L\vert_{y_3=M} + \frac{1}{2}\left|v^L\vert_{y_3=M}\right|^2\right)e_3\right),
$$
which is bounded for any $\delta>0$ by  
$$C \|\phi\|_{\cV}  E_{k+1} + \delta E_{k+1}+ C_\delta (\|\phi\|_{\cV}^2 +\|\phi\|_{\cV}^4 + \|\psi\|^2_{H^{m-\frac{1}{2}}_{uloc}}) (k+1)^2.$$ 
 There remains
$$
\int_{\Om_{bl}^M} \left((v^L + \tv)\cdot \na \tv\right)\cdot\begin{pmatrix}
0\\ \na_h \chi_k(y_h)\cdot \int_{\gamma(y_h)}^{y_3}\tv_h(y_h,z)dz
\end{pmatrix},
$$
which is bounded by $C(E_{k+1}-E_k)^{3/2} + C \|\phi\|_{\cV}(E_{k+1}-E_k)$. Gathering all the terms, we infer that for $\|\phi\|_{\cV}\leq 1$,
$$
E_k \leq C \left( (E_{k+1}-E_k)^{3/2} + (E_{k+1} - E_k) + \|\phi\|_{\cV} E_k + (\|\phi\|_{\cV}^2 + \|\psi\|^2_{H^{m-\frac{1}{2}}_{uloc}})(k+1)^2\right),
$$
where the constant $C$ depends only on $M$ and on $\|\gamma\|_{W^{1,\infty}}$. As a consequence, for $\|\phi\|_{\cV} $ small enough, we infer that for all $k\geq 1$,
$$
E_k \leq C \left( (E_{k+1}-E_k)^{3/2} + (E_{k+1} - E_k)  +  (\|\phi\|_{\cV}^2 + \|\psi\|^2_{H^{m-\frac{1}{2}}_{uloc}}) (k+1)^2\right).
$$
Thanks to a backwards induction argument (again, we refer to \cite{DGVNMnoslip} for all details), we infer that
$$
E_k \leq C  (\|\phi\|_{\cV}^2  + \|\psi\|^2_{H^{m-\frac{1}{2}}_{uloc}}) k^2\quad \forall k\in \N
$$
for a possibly different constant $C$. It follows that 
$$
\|\tv\|_{H^1_{uloc}(\Om^M_{bl})}\leq C  (\|\phi\|_{\cV} + \|\psi\|_{H^{m-\frac{1}{2}}_{uloc}})
$$
and therefore $v^-$ satisfies the same estimate. From there, we can derive a $L^2_{uloc}$ estimate for the pressure. 
Indeed, using the equation and the boundary condition at $y_3=M$, it follows that for all $y\in \Om_{bl}^M$,
$$
p^-(y_h, y_3)= \pa_3 v^-_3\vert_{y_3=M} - \frac{|v^-\vert_{y_3=M}|^2}{2} - \psi_3(y_h) - \int_{y_3}^M (\Delta v^-_3- v^-\cdot \na v^-_3)(y_h, z)\:dz.
$$
Note that by the divergence-free condition, the first-term in the right-hand side can be written as $-\dv_h v^-_h \vert_{y_3=M}$.
For $k\in \Z^2$, let $\varphi_k\in H^1_0(\Om_{bl}^M)$ such that $\Supp \:\varphi_k\subset (k+[0,1]^2)\times \R$. We multiply the above identity by $\varphi_k(x_h,z)$ and integrate over $\Om_{bl}^M$. After some integrations by parts, we obtain
\begin{eqnarray}
\int_{\Om_{bl}^M} p^- \varphi_k&=& \int_{\Om_{bl}^M} v^-_h\vert_{y_3=M} \cdot \na_h \varphi_k - \int_{\Om_{bl}^M} \frac{|v^-\vert_{y_3=M}|^2}{2} \varphi_k - \int_{\Om_{bl}^M} \psi_3 \varphi_k\label{pression1}\\
&-&\int_{\Om_{bl}^M} \left( \int_{y_3}^M (\Delta_h v^-_3 + \pa_3^2 v^-_3- v^-\cdot \na v^-_3)(y_h, z)\:dz\right) \varphi_k(y)\:dy.\label{pression2}
\end{eqnarray}
Using classical trace estimates and Sobolev embeddings, it follows that for all $q\in ]1,\infty[$,
$$
\| v^-\vert_{y_3=M}\|_{L^q_{uloc}(\R^2)}\leq C \| v^-\vert_{y_3=M}\|_{H^{1/2}_{uloc}(\R^2)}\leq C  \| v^-\|_{H^1_{uloc}(\Om_{bl}^M)}.
$$
Therefore the right-hand side of \eqref{pression1} is bounded by $C(\|\phi\|_{\cV} + \|\psi\|_{H^{m-1/2}_{uloc}}) \|\varphi_k\|_{H^1}$ for $\phi, \psi$ small enough. We now focus on \eqref{pression2}. The easiest term is the advection term: we have, since $\varphi_k$ has a bounded support (uniformly in $k$),
$$
\left| \int_{\Om_{bl}^M} \!\!\int_{y_3}^M v^-\cdot \na v^-_3(x_h, z)\:dz\; \varphi_k(y)\:dy\right| \leq C \|v^-\|_{L^4_{uloc}} \|\na v^-\|_{L^2_{uloc}} \|\varphi_k\|_{L^4}\leq C \|\varphi_k\|_{H^1}\|v^-\|_{H^1_{uloc}}^2.
$$
We then treat the two terms stemming from the laplacian separately. For the horizontal derivatives, we merely integrate by parts, recalling that $\varphi_k\in H^1_0(\Om_{bl}^M)$, so that
$$
\int_{\Om_{bl}^M}\int_{y_3}^M \Delta_h v^-_3(y_h,z)\:dz\; \varphi_k(y)\:dy=-\int_{\Om_{bl}^M}\int_{y_3}^M \na_h v^-_3(y_h,z)\cdot \na_h \varphi_k(y)\:dz\:dy,
$$
and the corresponding term is bounded by $C(\|\phi\|_{\cV} + \|\psi\|_{H^{m-1/2}_{uloc}}) \|\varphi_k\|_{H^1}$. As for the vertical derivatives, we have
\begin{eqnarray*}
&&\int_{\Om_{bl}^M}\left(\int_{y_3}^M\pa_3^2 v^-_3(y_h, z)\:dz\right)\varphi_k(y)\:dy\\&=& \int_{\Om_{bl}^M} \left(\pa_3 v^-_3(y_h,M) - \pa_3 v^-_3(y)\right)\varphi_k(y)\:dy\\
&=& - \int_{\Om_{bl}^M} (\na_h\cdot v^-_h(y_h,M) + \pa_3 v^-_3(y)) \varphi_k(y)\:dy\\
&=& \int_{\Om_{bl}^M}v^-_h(y_h,M) \cdot \na_h \varphi_k(y)\:dy - \int_{\Om_{bl}^M}  \pa_3 v^-_3(y)\varphi_k(y)\:dy.
\end{eqnarray*}
Both terms of the right-hand side are bounded by $C\| v^-\|_{H^1_{uloc}}\|\varphi_k\|_{H^1}$.

 Gathering the estimates of \eqref{pression1}, \eqref{pression2}, we infer that there exists a constant $C$ (independent of $\varphi_k$ and of $k$) such that for all $\varphi_k\in H^1_0(\Om_{bl}^M)$ supported in $(k+[0,1]^2)\times \R$,
 $$
 \left|\int_{\Om_{bl}^M} p^- \varphi_k \right|\leq C (\|\phi\|_{\cV} + \|\psi\|_{H^{m-1/2}_{uloc}}) \|\varphi_k\|_{H^1_0(\Om_{bl}^M)}.
 $$
 We deduce that 
 $$
 \|p^-\|_{H^{-1}_{uloc}(\Om_{bl}^M)}\leq  C (\|\phi\|_{\cV} + \|\psi\|_{H^{m-1/2}_{uloc}}) .
 $$
 Using the equation on $(v^-, p^-)$, we also have
 $$
 \|\na p^-\|_{H^{-1}_{uloc}(\Om_{bl}^M)}\leq  C (\|\phi\|_{\cV} + \|\psi\|_{H^{m-1/2}_{uloc}}).
 $$
 It then follows from Ne\u{c}as inequality (see \cite[Theorem IV.1.1]{BoyerFabrie}) that $p^-\in L^2_{uloc}(\Om_{bl}^M)$, with
 $$
 \|p^-\|_{L^2_{uloc}(\Om_{bl}^M)}\leq  C (\|\phi\|_{\cV} + \|\psi\|_{H^{m-1/2}_{uloc}}) .
 $$

\medskip
We still have to establish the two properties itemized in Lemma \ref{WPmoins}. We focus first on the higher order estimates. Note that using interior regularity results for the Stokes system (see \cite{Galdi}), one has $v^-\in H^N_{uloc}(\Om')$ for all open sets $\Om'\subset \R^2$ such that $\bar \Om'\subset \Om^M_{bl}$ and for all $N>0$. In particular, for all $M_1<M_2$ in the interval $] \sup \gamma, M[$,  $v^-\in H^{m+1}_{uloc}(\R^2\times (M_1,M_2)), p^-\in H^m_{uloc}(\R^2\times (M_1,M_2))$.

\medskip
We now tackle the regularity for $y_3>M'$, where $M'\in ] \sup \gamma, M[$. Our arguments are somehow standard (and mainly taken from \cite{BoyerFabrie}), but since there are a few difficulties related to the nonlinear stress boundary condition at $y_3=M$, we give details. The idea is to use an induction argument to show that $v^-\in H^l_{uloc}(\R^2\times [M',M])$ for all  $\sup \gamma<M'<M$ and for $1\leq l\leq m+1$. Unfortunately, the induction only works for $l\geq 2$: indeed, the implication $h\in H^s(\R^2)\Rightarrow h^2 \in H^s(\R^2)$, which is required to handle the nonlinear boundary condition at $y_3=M$, is true for $s>1$ only. Therefore we treat separately the case $l=2$. In the sequel, we write $\|\phi\| + \|\psi\|$ as a shorthand for $\|\phi\|_{\cV} + \|\psi\|_{H^{m-\frac{1}{2}}_{uloc}}$.

In order to prove $H^2_{uloc}$ regularity, the first step is to prove {\it a priori estimates} for $\pa_1 v^-, \pa_2 v^-$  in $H^1_{uloc}$. To that end, we first localize the equation near a fixed $k\in\Z^2$, then differentiate it with respect to $y_j$, $j=1,2$. Let $\theta\in \mathcal C^\infty_0(\R^2)$ be equal to 1 in a neighbourhood of $k\in \Z^2$, and such that the size of $\Supp\; \theta$ is bounded uniformly in $k$ (we omit the $k$-dependence of $\theta$ and of all subsequent functions in order to alleviate the notation). It can be easily checked that the equation satisfied by $w_j:=\pa_j (\theta v^-)$ is
$$
\ba -\Delta w_j + e_3\wedge w_j + v^-\cdot \na w_j + \na \pa_j (\theta p^-)=F_j \text{ in } \Om_\theta,\\
\dv w_j=g_j \text{ in } \Om_\theta,\\
w_j\vert_{y_3=M'}\in H^{1/2}(\R^2),\\
\left(\pa_3 w_j -(\pa_j (\theta p^-) + v^- \cdot w_j - \frac{1}{2} |v^-|^2 \pa_j \theta )e_3\right)\vert_{y_3=M}=\pa_j(\theta \psi),\\
w_j=0\text{ on }\pa \Supp \theta\times (M',M),\ea
$$
where $\Om_\theta  :=\Supp \theta \times (M',M)$ and 
$$\ba
F_j=\underbrace{\pa_j\left(-2 \na \theta \cdot \na v^- - v^- \Delta \theta + (v^-\cdot \na \theta) v^- + p^-\na \theta\right)}_{\| \cdot\|_{H^{-1}}\leq C(\|\phi\| + \|\psi\|) } - \pa_j v^- \cdot \na(\theta v^- ) ,\\
g_j=\pa_j(v^-\cdot \na \theta)= O(\|\phi\| + \|\psi\|)\text{ in } L^2(\R^2\times (M',M)).\ea$$
By standard results, see \cite[Section II.3]{Galdi}, there exists $\overline{w}_j \in H^1(\Omega_\theta)$ such that 
\begin{align*} 
&  \div \overline{w}_j = g_j, \quad \overline{w}_j = w_j  \quad   \mbox{at} \quad \pa \Omega_\theta \setminus \{y_3 = M\}, \\
&   \| \overline{w}_j \|_{H^1(\Omega_\theta)} \le C \left( \|  g_j \|_{L^2(\Omega_\theta)} + \| w_j \|_{H^{1/2}(\{y_3 = M'\})} \right)   
  \end{align*}
Note that we do not need to correct the trace of $w_j$ at $\{ y_3 = M \}$, as there is no Dirichlet boundary condition there. Moreover, we are not sure at this stage that this trace is a $H^{1/2}_{uloc}$ function. We rather prescribe an artificial smooth data for $\overline{w}_j$ at this boundary, chosen so that it satisfies the good compatibility condition. 
Finally,  $\tilde w_j = w_j-\bar w_j$ satisfies
$$
\ba -\Delta \tilde w_j + e_3\wedge \tilde w_j + v^-\cdot \na \tilde w_j + \na\tilde q_j=\tilde F_j \text{ in } \Om_\theta,\\
\dv \tilde w_j=0 \text{ in } \Om_\theta,\\
\tilde w_j\vert_{y_3=M'}=0,\quad \tilde w_j=0\text{ on }\pa \Supp \theta\times (M',M),,\\
\left(\pa_3 \tilde w_j -(\tilde q_j + v^- \cdot \tilde w_j  )e_3\right)\vert_{y_3=M}=\tilde \psi_j,\ea
$$
with $\tilde F_j=- \pa_j v^- \cdot \na(\theta v^- ) + O(\|\phi\| + \|\psi\|)$ in $H^{-1}$, and $\|\tilde \psi_j\|_{H^{-1/2}}\leq C(\|\phi\| + \|\psi\|)$. We obtain the estimate
$$
\|\na \tilde w_j\|_{L^2(\Om_\theta)}^2 \leq C(\|\phi\|^2 + \|\psi\|^2) + \left|\int_{\Om_{\theta}}\left(\pa_j v^- \cdot \na(\theta v^- )\right)\cdot\tilde w_j \right| + 2\int_{\Supp \theta}|v^-\vert_{y_3=M}| \: |\tilde w_j\vert_{y_3=M}|^2.
$$
We first deal with the boundary term:
\begin{eqnarray*}
\int_{\Supp \theta}|v^-\vert_{y_3=M}| \: |\tilde w_j\vert_{y_3=M}|^2&\leq & \|v^-\vert_{y_3=M}\|_{L^2(\Supp\theta)} \| \tilde w_j\vert_{y_3=M}\|_{L^4(\Supp \theta)}^2\\
&\leq & C \|v^-\vert_{y_3=M}\|_{H^{1/2}_{uloc}} \| \tilde w_j\vert_{y_3=M}\|_{H^{1/2}(\Supp \theta)}^2 
\leq  C\|v^-\|_{H^1}\| \tilde w_j\|_{H^1}^2\\&\leq &C (\|\phi\| + \|\psi\|)\|\na \tilde w_j\|_{L^2}^2.
\end{eqnarray*}
Hence for $\psi$ and $\phi$ small enough we can absorb this term in the left hand side of the energy inequality. As for the quadratic source term, we write
\begin{eqnarray*}
\pa_j v^- \cdot \na(\theta v^- )&=& \pa_j v^-_1 w_1 + \pa_j v^-_2 w_2 + \pa_j v^-_3 \theta \pa_3 v^-\\
&=&\pa_j v^-_1 w_1 + \pa_j v^-_2 w_2+ \pa_3 v^- w_{j,3} - v^-_3 \pa_j \theta \pa_3 v^-.
\end{eqnarray*}
For $i=1,..3$, $j=1,2$, $k=1,2$, we have
\begin{eqnarray*}
&&\int_{\Om_\theta} |\pa_i v^-| \:|w_j|\: |\tilde w_k|\leq C\|v^-\|_{H^1_{uloc}(\Om_{bl}^M)} \|w_j\|_{L^4(\Om_\theta)} \|\tilde w_k\|_{L^4(\Om_\theta)} \\
&\leq & C(\|\phi\| + \|\psi\|) (\|\tilde w_1\|_{H^1(\Om_\theta)}^2 + \|\tilde w_2\|_{H^1(\Om_\theta)}^2) + C(\|\phi\| + \|\psi\|)^3
\end{eqnarray*}
and
$$
\left| \int_{\Om_\theta}  v^-_3 \pa_j \theta \pa_3 v^-\cdot \tilde w_j\right|\leq C \|v^-_3\|_{H^1_{uloc}} \|\pa_3 v^-\|_{L^2_{uloc}} \|\tilde w_j\|_{H^1(\Om_\theta)}.
$$
Therefore, we obtain, for $\|\phi\|+ \|\psi\|$ small enough,
$$\|w_1\|_{H^1(\Om_\theta)}^2 + \|w_2\|_{H^1(\Om_\theta)}^2\leq C (\|\phi\|^2+ \|\psi\|^2).$$
Using the same idea as above to estimate $\pa_j(\theta p^-)$, this entails
$$
\| \na_h v^-\|_{H^1_{uloc}(\R^2\times(M',M))} + \|\na_h p^-\|_{L^2_{uloc}(\R^2\times (M',M))} \leq  (\|\phi\|_{\cV}+  \|\psi\|_{H^{m-\frac{1}{2}}_{uloc}}).
$$
Since $v^-$ is divergence free, similar estimates hold for $\pa_3 v_3^-$. Thus $v_3^-\in H^2_{uloc}(\R^2\times(M',M)) $. As for the vertical regularity of $v_h^-$, we observe that $\pa_3 v^-$ is a solution of the Stokes system with Dirichlet boundary conditions
$$
\ba
- \Delta \pa_3 v^- + \na \pa_3 p^-= F_3\text{ in } \R^2 \times (M',M),\\
\dv \pa_3 v^-=0 \text{ in } \R^2 \times (M',M),\\
\pa_3 v^-\vert_{y_3=M}= G,\\
\pa_3 v^-\vert_{y_3=M'}= G',
\ea
$$
where
$$
 F_3= - e_3\wedge \pa_3 v^- - \pa_3(v^-_h\cdot \na_h v^-) - \pa_3 (v^-_3\pa_3 v^-)\in H^{-1}_{uloc}(\R^2),\quad
G_h=\psi_h \in H^{m-1/2}_{uloc}(\R^2),
$$
and $G_3= \pa_3 v^-_3\vert_{y_3=M}\in H^{1/2}_{uloc}(\R^2)$, $G'\in H^{m-1/2}_{uloc}(\R^2)$.  Using the results of Chapter IV in \cite{Galdi}, we infer that $\pa_3 v^-\in H^1_{uloc}(\R^2\times (M',M))$, $\pa_3 p^- \in L^2_{uloc}(\R^2\times (M',M))$, and
\begin{multline*}
\|\pa_3 v^-\|_{ H^1_{uloc}(\R^2\times (M',M))} + \|\pa_3 p^-\|_{L^2_{uloc}(\R^2\times (M',M))}\\ \leq C (\|F\|_{H^{-1}_{uloc}} + \|G\|_{H^{1/2}_{uloc}} + \|G'\|_{H^{1/2}_{uloc}} )\leq C  (\|\phi\| + \|\psi\|)
\end{multline*}
for $\phi$ and $\psi$ small enough. Gathering the inequalities, we obtain
$$
\| v^-\|_{H^2_{uloc}(\R^2\times (M',M))} + \|p^-\|_{ H^1_{uloc}(\R^2\times (M',M))} \leq C  (\|\phi\|_{\cV} + \|\psi\|_{H^{m-\frac{1}{2}}_{uloc}}).
$$
Of course, all inequalities above are {\it a  priori} estimates,  but provide $H^2_{uloc}$ regularity  (and {\it a posteriori} estimates) through the usual method of translations. 

\medskip
We are now ready for the induction argument.  Let $k\in \Z^2$ be fixed. Define a sequence $\vartheta_k^2,\cdots \vartheta_k^{m+1}$ such that 
$\vartheta_k^l:=\theta_1^l(z-M) \theta_2^l(y_h-k)$, where $\theta_1^l\in \mathcal C^\infty_0(\R)$, $\theta_2^l\in  \mathcal C^\infty_0(\R^2)$ are equal to 1 in a neighbourhood of zero. We require furthermore that $\Supp \vartheta_k^{l+1}\subset (\vartheta_k^l)^{-1}(\{1\})$.
 We then define a $\mathcal C^{m+1,1}$ domain $\Om_k\subset \Om_{bl}^M$ such that $\Supp \vartheta_k^2\Subset \overline{\Om_k}$, and such that $\pa \Om_k\cap \pa \Om_{bl}=\emptyset$ (see Figure \ref{fig:Om_k}). Notice also that we choose $\Om_k$ so that $\mathrm{diam} (\Om_k)$ is bounded uniformly in $k$ (in fact, we can always assume that $\Om_k=(k,0)+ \Om_0$ for some fixed domain $\Om_0$.)
\begin{figure}
\caption{The domain $\Om_k$}
\begin{center}
\includegraphics[width=0.7\textwidth]{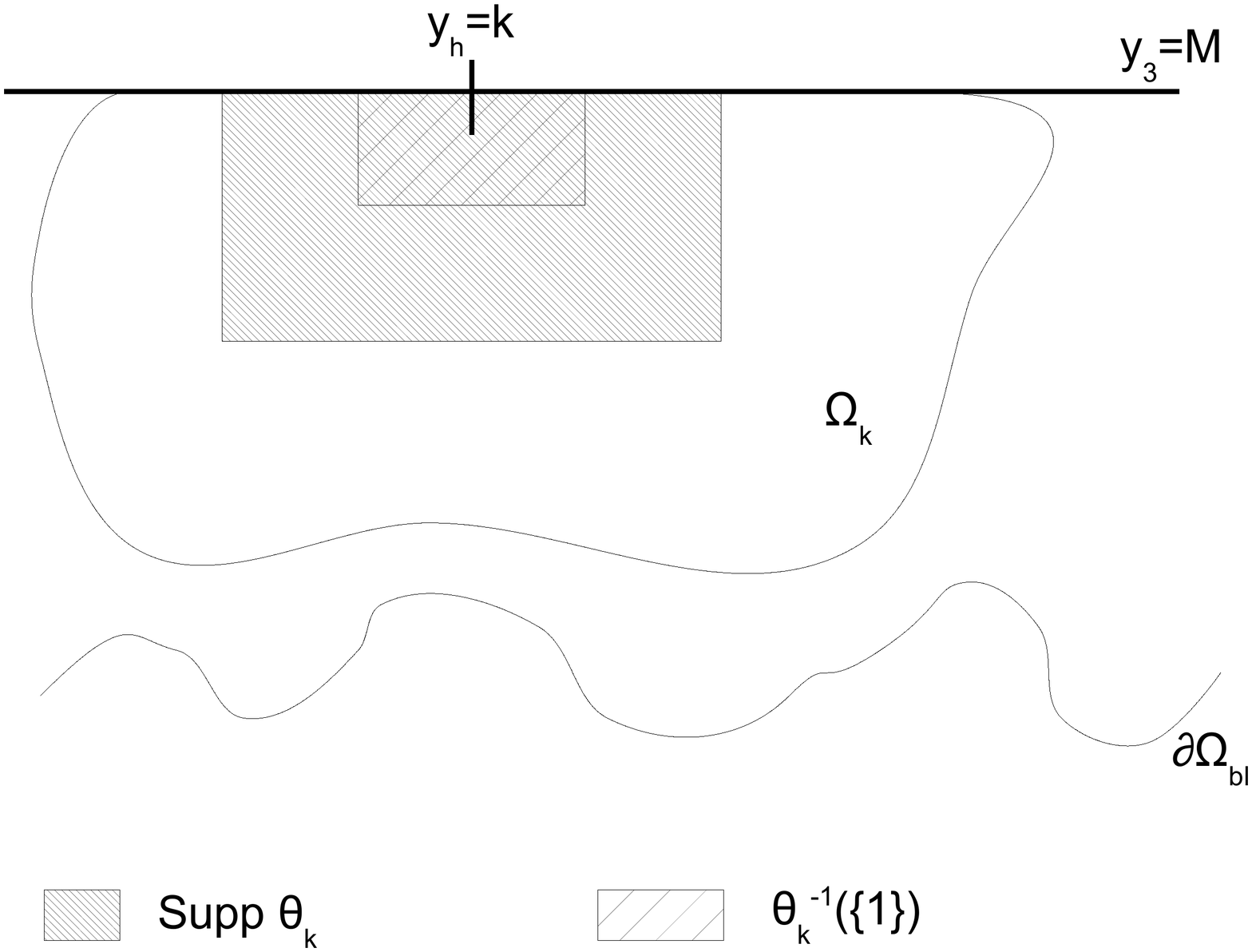}\label{fig:Om_k}
\end{center}
\end{figure}

Multiplying\eqref{NSC2} by  $\vartheta_k^l$ and dropping the dependence with respect to $k$, we find that $v^l\defin\vartheta_k^l v^-, p^l\defin p^- \vartheta_k^l$ is a solution of
\be\label{v}
\left\{\begin{array}{l}
- \Delta v^l + \na p^l= f^l\text{ in }\Om_k,\\
\dv v^l=g^l\text{ in }\Om_k,\\
\pa_n v^l - p^l n= \Sigma^l\text{ on }\pa \Om_k,
\end{array}
\right.
\ee
where 
\be\label{f}
\ba
f^l\defin - 2 \na \vartheta_k^l \cdot \na v^{l-1} - \Delta \vartheta_k^l v^{l-1} - (e_3\wedge v^{l-1} + v^{l-1}\cdot \na v^{l-1})\vartheta_k^l + p^{l-1}\cdot\na \vartheta_k,\\
g^l= v^{l-1}\cdot \na \vartheta_k,\\
\Sigma^l= \theta_2^l(y_h-k) \left(\psi + \frac{1}{2}|v^{l-1}|^2 e_3\vert_{y_3=M}\right)\text{ on }\pa \Om_k \cap \{y_3=M\},\\
\Sigma^l=0\text{ on } \pa \Om_k \cap \{y_3=M\}^c.
\ea
\ee
Now, Theorem  IV.7.1 in \cite{BoyerFabrie} implies that for all $l\in \{2,\cdots, m\}$, for $\|\phi\|_{\cV}  + \|\psi\|_{H^{m-\frac{1}{2}}_{uloc}}$ small enough,
$$
(v^l, p^l)\in H^l(\Om_k)\times H^{l-1}(\Om_k)\Rightarrow (v^{l+1}, p^{l+1})\in H^{l+1}(\Om_k)\times H^{l}(\Om_k),
$$
and
$$
\|v^{l+1}\|_{H^{l+1}(\Om_k)} + \| p^{l+1}\|_{H^{l}(\Om_k)} \leq C \left( \|v^l\|_{H^{l}(\Om_k)} + \| p^l\|_{H^{l-1}(\Om_k)} +\|\psi\|_{H^{l-1/2}(\Om_k)}\right).
$$
Indeed, assume that $(v^l, p^l)\in H^l(\Om_k)\times H^{l-1}(\Om_k)$. Then  $f^{l+1}\in H^{l-1}(\Om_k)$, $g^{l+1}\in H^{l}(\Om_k)$,
with
$$
\|f^{l+1}\|_{H^{l-1}(\Om_k)}\leq C(\|v^l\|_{H^l} + \|v^l\|_{H^l}^2 + \| p^l\|_{H^{l-1}(\Om_k)}),\ \|g^{l+1}\|_{H^l}\leq C \|v^l\|_{H^l}.
$$
Moreover, $v^l\in H^{l-\frac{1}{2}}(\pa \Om_k)$. Since $l\geq 2$, using product laws in fractional Sobolev spaces (see \cite{strichartz}), we infer that $|v^l|^2\vert_{|y_3=M}\in H^{l-\frac{1}{2}}(\R^2)$, and therefore $\Sigma^{l+1}\in  H^{l-\frac{1}{2}}(\R^2)$. From \cite[Theorem  IV.7.1]{BoyerFabrie}, we deduce that  $(v^{l+1}, p^{l+1})\in H^{l+1}(\Om_k)\times H^{l}(\Om_k)$, together with the announced estimate. By induction $v^-\in H^{m+1}_{uloc}(\Om_{bl}^M), p^-\in H^{m}_{uloc}(\Om_{bl}^M)$.  

\smallskip

There only remains to check the compatibility condition at $y_3=M$. Notice that
\begin{eqnarray*}
v_3^-\vert_{|y_3=M}&=&\phi_3+\int_{\gamma(y_h)}^M \pa_3 v^-_3= \phi_3 - \int_{\gamma(y_h)}^M \na_h \cdot v_h^-\\
&=&\phi_3 - \gamma(y_h)\cdot \phi_h +\na_h\cdot \nu_h,
\end{eqnarray*}
where
$$
\nu_h=-\int_{\gamma(y_h)}^M v_h^-  \in H^{1/2}_{uloc}(\R^2).
$$
Since $\phi_3 - \gamma(y_h)\cdot \phi_h$ due to the non-penetrability condition $\phi\cdot n=0$, we obtain the desired identity. 
\end{proof}

\medskip
{\em Step 2}.   Once $(v^-, p^-)$ is defined thanks to Lemma \ref{lem:v-}, we define $(v^+, p^+)$ in the half-space $\{y_3>M\}$ by solving \eqref{NSC1} with $v^+\vert_{y_3=M}=v^-\vert_{y_3=M}$. According to Lemma \ref{lem:v-} and to standard trace inequalities,
$$
\|v^-\vert_{y_3=M}\|_{H^{m+\frac{1}{2}}_{uloc}(\R^2)}\leq C  (\|\phi\|_{\cV} +  \|\psi\|_{H^{m-\frac{1}{2}}_{uloc}}),
$$
for some constant $C$ depending only on $M$ and on $\|\gamma\|_{W^{1,\infty}}$. As a consequence, if $C  (\|\phi\|_{\cV} +\|\psi\|_{H^{m-\frac{1}{2}}_{uloc}}) + \|\nu_h\|_{L^2_{uloc}}\leq \delta_0$, according to Proposition \ref{prop:NS-1} the system \eqref{NSC1} with $v_0=v^-_{|y_3=M}$ has a unique solution.

\medskip
Additionally, $\Sigma(v^+, p^+) e_3\vert_{|y_3=M^+}$ belongs to $H^{m-\frac{1}{2}}_{uloc}(\R^2)$. Thus the mapping
$$
\cF:\begin{array}{ccc}
\cV\times H^{m-\frac{1}{2}}_{uloc} (\R^2)&\longrightarrow & H^{m-\frac{1}{2}}_{uloc}(\R^2)\\
(\phi, \psi)&\mapsto&  \Sigma(v^+, p^+) e_3\vert_{|y_3=M^+} - \psi
\end{array}
$$
is well-defined. Clearly, according to Lemma \ref{lem:v-}, for $\phi=0$ and $\psi=0$, we have $v^-=0$, $v^+=0$ and therefore $\cF(0,0)=0$.

\medskip
 The strategy is then to apply the implicit function theorem to $\cF$ to find a solution of $\cF(\phi, \psi)=0$ for $\phi$ in a neighbourhood of zero. Therefore we check that $\cF$ is $\mathcal C^1$ in a neighbourhood of zero, and that its Fr\'echet derivative with respect to $\psi$ at $(0,0)$ is an isomorphism on $H^{m-\frac{1}{2}}_{uloc}(\R^2)$.
 
 \begin{itemize}
 \item \textit{ $\cF$ is a $\mathcal C^1$ mapping in a neighbourhood of zero:}
 \end{itemize}
Let $\phi_0, \psi_0$ (resp. $\phi,\psi$) in a neighbourhood of zero (in the sense of the functional norms in $\cV$ and $H^{m-\frac{1}{2}}_{uloc} (\R^2)$). We denote by $v_0^\pm, p_0^\pm, v^\pm, p^\pm$ the solutions of \eqref{NSC1}, \eqref{NSC2} associated with $(\phi_0,\psi_0)$, $(\phi_0+\phi, \psi_0+ \psi)$ respectively, and we set $w^\pm:= v^\pm-v_0^\pm$, $q^\pm= p^\pm-p_0^\pm$. 

On the one hand,  in $\Om_{bl}^M$, $w^-$ is a solution of the system
$$
\ba
- \Delta w^- +e_3\wedge w^-+  (v_0^-+ w^-)\cdot \na w^- + w^-\cdot \na v_0^-+ \na q^-=0,\\
\dv w^-=0,\\
w^-\vert_{\pa \Om_{bl}}=\phi,\\
\left[\pa_3 w^- - q^- e_3 - \frac{2v_0^-\cdot w^- + |w^-|^2}{2}e_3\right]\vert_{y_3=M}=\psi.
\ea
$$
Performing estimates similar to the ones of Lemma \ref{lem:v-}, we infer that for $\|\phi_0\|_{\cV} + \|\psi_0\|_{H^{m-1/2}_{uloc}}$ and $ \|\phi\|_{\cV} + \|\psi\|_{H^{m-1/2}_{uloc}}$ small enough,
$$
\| w^- \|_{H^1_{uloc}(\Omega_{bl})}  +     \|w^-\vert_{y_3=M}\|_{H^{m+1/2}_{uloc}}\leq C (\|\phi\|_{\cV} + \|\psi\|_{H^{m-1/2}_{uloc}}). 
$$
It follows that $w^-= w_L^- + O(\|\phi\|_{\cV}^2 + \|\psi\|_{H^{m-1/2}_{uloc}}^2)$ in $H^1_{uloc}(\Om_{bl}^M)$ and in $H^{m+1}_{uloc}((M',M)\times \R^2)$, for all $M'>\sup \gamma$, where $w_L^-$ solves the same system as $w^-$ minus the quadratic terms $w^-\cdot \na w^-$ and $|w^-|^2\vert_{y_3=M}$. 

\medskip
On the other hand, using Theorem \ref{thm_SCsource}, one can show that 
, $w^+=w_L^+ + O(\|\phi\|_{\cV}^2 + \|\psi\|_{H^{m-1/2}_{uloc}}^2)$, where 
$$
\ba 
- \Delta w^+_L +e_3\wedge w^++ v_0^+\cdot \na w^+_L + w^+_L\cdot \na v_0^++ \na q^+_L=0\text{ in }y_3>M,\\
\dv w^+_L=0\text{ in }y_3>M,\\
w^+_L\vert_{y_3=M}=w^-_L\vert_{y_3=M}.
\ea
$$
Using Theorem \ref{thm_SCsource}, we deduce that if $\|(1+y_3)^{1/3} v_0^+\|_{H^{m+1}_{uloc}}$ is small enough (which is ensured by the smallness condition on $\|\phi\|, \|\psi\|$), we have
$$
\|(1+y_3)^{1/3} w_L^+\|_{H^{m+1}_{uloc}(\R^3_+)}\leq C \|w_L^-\vert_{y_3=M}\|_{H^{m+1/2}_{uloc}}\leq C(\|\phi\|_{\cV} + \|\psi\|_{H^{m-1/2}_{uloc}}). 
$$
Therefore, in $H^{m-\frac{1}{2}}_{uloc}(\R^2)$,
\begin{multline*}
\cF(\phi_0+\phi, \psi_0+ \psi) - \cF(\phi_0, \psi_0)\\= -\psi + \pa_3 w_L^+\vert_{y_3=M} - \left(q^+_L +  v_0^+\cdot w^+_L\right)_{|y_3=M} e_3+ O(\|\phi\|_{\cV}^2 + \|\psi\|_{H^{m-1/2}_{uloc}}^2).
\end{multline*}

It follows that the Fr\'echet derivative of $\cF$ at $(\phi_0, \psi_0)$ is
$$
\mathcal L_{\phi_0, \psi_0}:(\phi,\psi)\mapsto -\psi + \pa_3 w_L^+\vert_{y_3=M} - \left(q^+_L + v_0^+\cdot w^+_L\right)_{|y_3=M}e_3.
$$
Using the same kind of  arguments as above, it is easily proved that $w_L^\pm$ depend continuously on $v_0^\pm$, and therefore on $\phi_0, \psi_0$. Therefore $\cF$ is a $\mathcal C^1$ function in a neighbourhood of zero.

\medskip
\begin{itemize}
 \item \textit{ $d_{\psi}\cF(0,0)$ is invertible:}
\end{itemize}
Since $d_{\psi}\cF(0,0)= \mathcal L_{0,0}(0,\cdot)$, we consider the systems solved by $w_L^\pm$ with $v_0^\pm=0$ and $\phi=0$. 
We first notice that if $\mathcal L_{0,0}(0,\psi)=0$, then $w_L:=\mathbf 1_{y_3\leq M } w_L^- + \mathbf 1_{y_3>M} w_L^+$ is a solution of the Stokes-Coriolis system in the whole domain $\Om_{bl}$, with $w_L\vert_{\pa \Om_{bl}}=0$. Therefore, according to \cite{DalibardPrange}, $w_L\equiv 0$ and therefore $\psi=0$. Hence $\ker d_{\psi}\cF(0,0)=\{0\}$, and $d_{\psi}\cF(0,0)$ is one-to-one.

On the other hand,
$$
\left(\pa_3 w_L^+ - q_L^+ e_3\right)\vert_{y_3=M}=DN(w_L^-\vert_{y_3=M}),
$$
where $DN$ is the Dirichlet-to-Neumann operator for the Stokes-Coriolis system, introduced in \cite{DalibardPrange}. In particular, in order to solve the equation
$$
\mathcal L_{0,0}(0,\psi_1)=\psi_2,
$$ 
for a given $\psi_2\in H^{m-1/2}_{uloc}(\R^2)$, we need to solve the system
$$
\ba
- \Delta w_L^- +e_3\wedge w_L^-  + \na q_L^-=0,\\
\dv w_L^-=0,\\
w_L^-\vert_{\pa \Om_{bl}}=0,\\
\left[\pa_3 w_L^- - q_L^- e_3 - \right]\vert_{y_3=M}=-\psi_2 + DN(w_L^-\vert_{y_3=M}).
\ea
$$
According to section 3 in \cite{DalibardPrange}, the above system has a unique solution $w_L^-\in H^1_{uloc}(\Om_{bl}^M)$. There only remains to prove  that $w_L^-\in H^{m+1}_{uloc}(\{M' < y_3 < M\})$ for all $\sup \gamma<M'<M$. Therefore, we notice that  in the domain $\R^2\times (M',M)$, the horizontal derivatives of $w_L^-$ (up to order $m$) satisfy a Stokes-Coriolis system similar to the one above (notice that the Dirichlet-to-Neumann operator commutes with $\pa_1$, $\pa_2$). It follows that $\na_h^\alpha w_L^-\in H^1_{uloc}(\R^2\times (M',M))$ for all $|\alpha|\leq m$. In particular, $\na_h^\alpha w_L^-\vert_{|y_3=M}\in H^{1/2}_{uloc}(\R^2)$ and therefore $w_L^-\vert_{|y_3=M}\in H^{m+1/2}_{uloc}(\R^2)$. It can be checked that $DN: H^{m+1/2}_{uloc}(\R^2) \to H^{m-1/2}_{uloc}(\R^2)$. As a consequence, $\psi_1=\pa_3 w_L^- - q_L^- e_3 \in H^{m-1/2}_{uloc}(\R^2)$. Therefore $d_{\psi}\cF(0,0)$ is an isomorphism of $H^{m-1/2}(\R^2)$.

\vskip2mm

Using the implicit function theorem, we infer that for all $\phi\in \cV$ in a neighbourhood of zero, there exists $\psi\in H^{m-1/2}_{uloc}(\R^2)$ such that $\cF(\phi,\psi)=0$. Let $v:=\mathbf 1_{y_3\leq M} v^- + \mathbf 1_{y_3>M} v^+$, where $v^-, v^+$ are the solutions of \eqref{NSC2}, \eqref{NSC1} associated with $\phi, \psi$. By definition, the jump of $v$ across $\{y_3=M\}$ is zero, and since $\cF(\phi,\psi)=0$,
$$
\Sigma(v^-,p^-)e_3\vert_{y_3=M}=\psi=\Sigma(v^+,p^+)e_3\vert_{y_3=M}.
$$
Using once again the fact that $|v^+|^2\vert_{y_3=M}=|v^-|^2\vert_{y_3=M}$, we deduce that
$$
\left(\pa_3 v^- - p^- e_3\right)\vert_{y_3=M}=\left(\pa_3 v^+ - p^+ e_3\right)\vert_{y_3=M}.
$$
 Thus there is no jump of the stress tensor across $\{y_3=M\}$, and therefore $v$ is a solution of the Navier-Stokes-Coriolis system in the whole domain $\Om_{bl}$. This concludes the proof of Theorem \ref{Thm_WP}.

\section*{Acknowledgements}

The authors have been partially funded by the ANR  project Dyficolti ANR-13-BS01-0003-01. Anne-Laure Dalibard  has received funding from the European Research Council (ERC) under the European Union’s Horizon 2020 research and innovation program (Grant agreement No 63765).

  \section*{Appendix A: Proof of Lemma \ref{lemmadecay} and Lemma \ref{lemmadecay2}}
  
\subsection*{Proof of Lemma \ref{lemmadecay}.}
We begin with a few observations. First, replacing $\chi$ by $\chi_1:= Q \chi \in \mathcal{C}_c^\infty(\R^2)$, it is enough to prove the lemma with $Q=1$. Moreover it is clearly sufficient to prove the lemma for $p_k(\xi)= \xi_1^a \xi_2^b$, with $a+b=k$. Notice also that since $\alpha-k\ge-2$, we can always write $\alpha - k= 2m + \alpha_m$, with $\alpha_m\in [-2, 0[$ and $m\in \N$. Then $\xi_1^a \xi_2^b |\xi|^{\alpha-k}$ is a linear combination of terms of the form $\xi_1^{a'} \xi_2^{b'} |\xi|^{\alpha_m}$, with $a'+b'+\alpha_m=\alpha$ and $a',b'\in \N$. Therefore, in the rest of the proof, we take
$$
Q\equiv 1,\quad P(\xi)= \xi_1^a\xi_2^b |\xi|^\beta,\text{ with }a, b \in \N,\ \beta \in  [-2, 0[, a+b+\beta=\alpha.
$$

Some of the arguments of the proof are inspired by the work of Alazard, Burq and Zuily \cite{ABZ} on the Cauchy problem for gravity water waves in $H^s_{uloc}$ spaces. We  introduce a partition of unity $(\varphi_q)_{q\in \Z^2}$, where $\Supp \varphi_q \subset B(q, 2)$ for $q\in \Z^2$, $\sup_q \|\varphi_q\|_{W^{k,\infty}} <+\infty$ for all $k$. We also introduce functions $\tilde \varphi_q\in \mathcal C^\infty_0(\R^2)$ such that $\tilde \varphi_q\equiv 1$ on $\Supp \varphi_q$, and, say $\Supp \tilde\varphi_q \subset B(q, 3)$.
Then, for $j=1,2,3$,
\begin{eqnarray}
u^j(x_h, z)&=&\sum_{q\in \Z^2}\chi(D)P(D)e^{-\lambda_j(D)z} (\varphi_q u_0)\nonumber\\
&=&\sum_{q\in \Z^2} \int_{\R^2}K^j(x_h-y_h, z) \varphi_q(y_h) u_0(y_h)\:dy_h\nonumber\\
&=& \sum_{q\in \Z^2} \int_{\R^2}K^j_q(x_h, y_h, z) \varphi_q(y_h) u_0(y_h)\:dy_h,\label{dec-ui}
\end{eqnarray}
where
$$
K^j(x_h, z)= \int_{\R^2}e^{i x_h\cdot \xi } \chi(\xi) P(\xi) e^{-\lambda_j(\xi )z}\:d\xi,\ 
K^j_q(x_h, y_h, z)=K^j(x_h-y_h, z)\tilde \varphi_q(y_h).
$$
We then claim that the following estimates hold: there exists $\delta>0$, $C\geq 0$ such that for all $x_h\in \R^2, z>0$,
\be\label{est:K}
\ba
|K^1(x_h,z)| \leq \frac{C}{(1+|x_h| + z^{1/3})^{2+\alpha}},\\
|K^j(x_h,z)|\leq C \frac{e^{-\delta z}}{(1+ |x_h|)^{2+\alpha}}\quad \text{for }j=2,3.
\ea
\ee

Let us postpone the proof of estimates \eqref{est:K} and explain why Lemma \ref{lemmadecay} follows.  Going back to \eqref{dec-ui}, we have, for $j=2,3$,
\begin{eqnarray*}
|u^j(x_h, z)| &\leq & C e^{-\delta z} \sum_{q\in \Z^2,|q-x_h|\geq 3} \frac{1}{(|q-x_h|-2)^{2+\alpha} } \int |\varphi_q(y_h) u_0(y_h)| dy_h  \\
&+&  C e^{-\delta z} \sum_{q\in \Z^2,|q-x_h|\leq 3}  \int |\varphi_q(y_h) u_0(y_h)| dy_h \\
&\leq & C e^{-\delta z }\|u_0\|_{L^1_{uloc}}.
\end{eqnarray*}
In a similar fashion,
\begin{eqnarray*}
|u^1(x_h, z)| &\leq & C\sum_{q\in \Z^2,|q-x_h|\geq 3 } \frac{1}{(|q-x_h| - 2 + z^{1/3})^{2+\alpha}} \int |\varphi_q(y_h) u_0(y_h)| dy_h \\
&+&  C\sum_{q\in \Z^2,|q-x_h|\leq 3}\frac{1}{(1+z^{1/3})^{2+\alpha}}\int |\varphi_q(y_h) u_0(y_h)| dy_h\\
&\leq & C \|u_0\|_{L^1_{uloc}} (1+z)^{-\alpha/3}.
\end{eqnarray*}
The estimates of Lemma \ref{lemmadecay} follow for $z\geq 1$.

\medskip
We now turn to the proof of estimates \eqref{est:K}. Once again we start with the estimates for $K^2, K^3$, which are simpler. Since $\lambda_2, \lambda_3$ are continuous and have non-vanishing real part on the support of $\chi$, there exists a constant $\delta>0$ such that $\Re(\lambda_j(\xi))\geq \delta$ for all $\xi \in \Supp \chi$ and for $j=2,3$. Clearly, for $|x_h|\leq 1$ we have simply
$$
|K^j(x_h,z)| \leq e^{-\delta z} \|\chi P\|_{L^1}.
$$
We thus focus on the set $|x_h|\geq 1$.
 Let $\chi_j(\xi, z)\defin \chi(\xi)\exp(-\lambda_j(\xi )z)$. Then $\chi_j\in L^\infty(\R_+, \mathcal S(\R^2))$, and for all $n_1, n_2, n_3\in \N$, there exists a constant $\delta_n>0$ such that 
 $$
\left| (1+|\xi|^{n_3}) \pa_1^{n_1} \pa_2^{n_2}  \chi_j(\xi, z)\right| \leq C_n \exp(-\delta_n z). 
 $$
Estimate \eqref{est:K} for $K^2, K^3$ then follows immediately from the following Lemma (whose proof is given after the current one):
\begin{lemma} \label{lemma8}
Let $P(\xi)= \xi_1^{a_1}\xi_2^{a_2} |\xi|^\beta$, with $a_1, a_2\in \N$, $\beta\in [-2,0[$, and set $\alpha:= a_1+a_2+ \beta$.
Then there exists  $C>0$ such that: for any  $\zeta\in \mathcal{S}(\R^2)$, for all $x_h\in \R^2$, $|x_h|\geq 1$,
$$
\left|P(D) \zeta (x_h)\right|\leq \frac{C}{|x_h|^{2+\alpha}} \left(\|\zeta\|_1  
+ \| |y_h|^{a_1+a_2+2}\pa_1^{a_1}\pa_2^{a_2} \zeta\|_\infty   \right).
$$
\label{lem:op-integral}
\end{lemma}
We now address the estimates on $K^1$. When $|x_h|\leq 1$, $z\leq 1$, we have simply $|K^1(x_h,z)|\leq \|P\chi\|_1$, and the estimate follows. When $z\leq1$ and $|x_h|\geq 1$ we apply Lemma \ref{lem:op-integral} with $\zeta(\xi)=\mathcal F^{-1}\left(\chi(\xi) \exp(-\lambda_1(\xi) z)\right)$. Notice that
$$
\| \zeta\|_1   \lesssim \|\chi(\xi) \exp(-\lambda_1(\xi) z)\|_{W^{3,1}}, 
$$ 
and 
$$ \| |y_h|^{a_1+a_2+2}\pa_1^{a_1}\pa_2^{a_2} \zeta\|_\infty
\lesssim \|\xi_1^{a_1}\xi_2^{a_2} \chi(\xi) \exp(-\lambda_1(\xi) z)\|_{W^{2+a_1+a_2,1}}.
$$
Since  the right-hand sides of the above inequalities are bounded (recall that $\lambda_1(\xi)=|\xi|^3 \Lambda_1(\xi)$ with $\Lambda_1\in \mathcal C^\infty(\R^2)$, see Remark \ref{remark_roots}), it follows that estimate \eqref{est:K} is true for $z\leq1$ and $|x_h|\geq 1$. 

We now focus on the case $z\geq 1$. We first change variables in the integral defining $K^1$ and we set $\xi'= z^{1/3} \xi$, $x_h'= \frac{x_h}{z^{1/3}}.$ Since $P$ is homogeneous, this leads to
$$
K^1(x_h, z)= \frac{1}{z^{\frac{2+\alpha}{3}}}\int_{\R^2} e^{i x_h'\cdot \xi'} P(\xi') \chi\left(\frac{\xi'}{z^{1/3}}\right) \exp\left(- \lambda_1\left(\frac{\xi'}{z^{1/3}}\right) z\right) d\xi'.
$$ 
 Since $\lambda_1/|\xi|^3$ is continuous and does not vanish on the support of $\chi$, there exists a positive constant $\delta'$ such that $\lambda_1(\xi )\geq \delta' |\xi|^3$ on $\Supp \chi$. Therefore, for $|x_h'|\leq 1$, we have
$$
|K^1(x_h,z)|\leq \frac{1}{z^{\frac{2+\alpha}{3}}}\left\| \exp(-\delta' |\xi|^3) P(\xi')\right\|_{L^1},
$$
and the estimate for $K^1$ on the set $|x_h| \leq z^{1/3}$ is proved.

For $|x_h'|\geq 1$, we split the integral in two. Let $\varphi \in \mathcal C^\infty_0$ such that $\varphi\equiv 1$ in a neighbourhood of zero. Then
\begin{eqnarray*}
K^1(x_h, z)&=& \frac{1}{z^{\frac{2+\alpha}{3}}}\int_{\R^2} e^{i x_h'\cdot \xi'} P(\xi') \varphi(\xi')\chi\left(\frac{\xi'}{z^{1/3}}\right) \exp\left(- \lambda_1\left(\frac{\xi'}{z^{1/3}}\right) z\right) d\xi'\\
&&+ \frac{1}{z^{\frac{2+\alpha}{3}}}\int_{\R^2} e^{i x_h'\cdot \xi'} P(\xi') (1-\varphi(\xi'))\chi\left(\frac{\xi'}{z^{1/3}}\right) \exp\left(- \lambda_1\left(\frac{\xi'}{z^{1/3}}\right) z\right) d\xi'\\
&=:& K^1_1 + K^1_2.
\end{eqnarray*}
We first consider the term $K^1_2$. Because of the truncation $1-\varphi$, we have removed all singularity coming from $P$ close to $\xi=0$. Therefore, performing integrations by part, we have, for any $n\in \N$, for $j=1,2$,
$$
{x_j'}^nK^1_2(x_h, z)= \frac{1}{z^{\frac{2+\alpha}{3}}}\int_{\R^2} e^{i x_h'\cdot \xi'}D_{\xi'_j}^n \left[ P(\xi') (1-\varphi(\xi'))\chi\left(\frac{\xi'}{z^{1/3}}\right) \exp\left(- \lambda_1\left(\frac{\xi'}{z^{1/3}}\right) z\right)\right]\: d\xi'.$$
When the $D_{\xi'_j}$ derivative hits $P(1-\varphi)$, we end up with an integral bounded by
$$
C_n \int_{\R^2} |\xi'|^\alpha \mathbf 1_{\xi'\in \Supp (1-\varphi)}\exp(-\delta'|\xi'|^3)\:d\xi'\leq C_n.
$$
When the derivative hits $\chi\left(\frac{\xi'}{z^{1/3}}\right)$ the situation is even better as a power of $z^{1/3}$ is gained with each derivative. Therefore the worst terms occur when the derivative hits the exponential. Remember that $\lambda_1(\xi)= |\xi|^3 \Lambda_1(\xi)$, where $\Lambda_1\in \mathcal C^\infty(\R^2)$ with $\Lambda_1(0)=1$ and $\Lambda_1$ does not vanish on $\R^2$. Therefore, for all $\xi'\in \R^2$, $z>0$,
$$
\exp\left(- \lambda_1\left(\frac{\xi'}{z^{1/3}}\right) z\right)= \exp\left(- |\xi'|^3 \Lambda_1\left(\frac{\xi'}{z^{1/3}}\right)\right).
$$
We infer that for any $0\leq n \leq 3+ \lfloor \alpha \rfloor$, on $\Supp \chi(\cdot/z^{1/3})$, we have
\be\label{est:der-exp}
\left| P(\xi')\na_{\xi'_j}^n\exp\left(- \lambda_1\left(\frac{\xi'}{z^{1/3}}\right) z\right)\right| \leq C_n \exp\left(-\frac{\delta'}{2}|\xi|^3\right).
\ee
We deduce eventually that
$$
|K^1_2(x_h, z)| \leq C \frac{1}{z^{\frac{2+\alpha}{3}}} \frac{1}{(1+ |x_h'|^{2+\alpha})} \leq \frac{C}{(|x_h|+ z^{1/3})^{2+\alpha}}.
$$

For the term $K^1_1$, we use once again Lemma \ref{lem:op-integral}, with 
$$
\zeta\defin \mathcal F^{-1}\left( \varphi(\xi')\chi\left(\frac{\xi'}{z^{1/3}}\right) \exp\left(- \lambda_1\left(\frac{\xi'}{z^{1/3}}\right) z\right) \right).
$$
Using the same type of estimate as \eqref{est:der-exp} above, we obtain
$$
|K^1_2(x_h, z)| \leq C \frac{1}{z^{\frac{2+\alpha}{3}}} \frac{1}{|x_h'|^{2+\alpha}} \leq \frac{C}{|x_h|^{2+\alpha}}.
$$
This concludes the proof of  Lemma \ref{lemmadecay}.\qed

\begin{proof}[Proof of Lemma \ref{lem:op-integral}]
We have
$$
P(D)\zeta= D_1^{a_1} D_2^{a_2} \mathrm{Op}(|\xi|^{\beta}) \zeta.
$$
Thus we first compute $\mathrm{Op}(|\xi|^{\beta}) \zeta$. 
We first focus on the case $\beta\in ]-2,0[$. We follow the ideas of Droniou and Imbert \cite[Theorem 1]{DroniouImbert}, recalling the main steps of the proof. The function $\xi\in \R^2\mapsto |\xi|^{\beta}$ is radial and locally integrable, and thus belongs to $\mathcal S'$. Its Fourier transform in $\mathcal S'(\R^2)$ is also radial and homogeneous of degree $-\beta-2 \in  ]-2,0[$. Therefore it coincides (up to a constant) with $|\cdot|^{-\beta-2}$ in $\mathcal S'(\R^2\setminus\{0\})$, and since the latter function is locally integrable, we end up with $\mathcal F^{-1}(|\xi|^\beta)=C |x_h|^{-\beta-2}$ in $\mathcal S'(\R^N)$. Hence
$$
P(D)\zeta(x_h)= C \pa_1^{a_1} \pa_2^{a_2} \int_{\R^2} \frac{1}{|y_h|^{\beta+2}} \zeta( x_h-y_h)\:dy_h.
$$
Notice that in the present case, we do not need to have an exact formula for $P(D)\zeta$, but merely some information about its decay at infinity. As a consequence we take a short-cut in the proof of \cite{DroniouImbert}. We take a cut-off function $\chi\in \mathcal C^\infty_0(\R^2)$ such that $\chi\equiv 1$ in a neighbourhood of zero, and we write
\begin{eqnarray*}
P(D)\zeta(x_h)&=& C\int_{\R^2} \frac{\chi(y_h)}{|y_h|^{\beta+2}} \pa_1^{a_1}\pa_2^{a_2} \zeta(x_h-y_h)\:dy_h\\
&+& C \sum_{\substack{0\leq i_1\leq a_1,\\0\leq i_2\leq a_2}}C_{i_1,i_2}\int_{\R^2}\pa_1^{i_1}\pa_2^{i_2}(1-\chi(y_h)) \pa_1^{a_1-i_1}\pa_2^{a_2-i_2} \left(\frac{1}{|y_h|^{\beta+2}}\right)\zeta(x_h-y_h)\:dy_h\\
&=:&I_1+I_2.
\end{eqnarray*}
We now choose $\chi$ in the following way. Let $n=\lfloor |x_h|\rfloor\in \N$, and take $\chi= \chi_n=\eta(\cdot/n)$, where $\Supp \eta \subset B(0,1/2)$ and $\eta\equiv 1$ in a neighbourhood of zero. Notice in that case that if $y_h\in \Supp \chi_n$, then $|x_h-y_h|\geq |x_h|/2$. Therefore, for the first term, we have
\begin{eqnarray*}
|x_h|^{2+\alpha} |I_1| &\leq& (n+1)^\beta \left(\int_{|y_h|\leq n/2} |y_h|^{-\beta-2}\:dy_h\right) \left\| |y_h|^{2+a_1+a_2} \pa_1^{a_1}\pa_2^{a_2} \zeta\right\|_{L^\infty}\\&\leq& C \left\| |y_h|^{2+a_1+a_2} \pa_1^{a_1}\pa_2^{a_2} \zeta\right\|_{L^\infty}.
\end{eqnarray*}
Using the assumptions on $\eta$ and $\chi_n$ and the estimate
$$
\left| \pa_1^{a_1-i_1}\pa_2^{a_2-i_2} \left(\frac{1}{|y_h|^{\beta+2}}\right)\right|\leq \frac{C}{|y_h|^{\alpha+2-i_1-i_2}}\leq \frac{C}{n^{\alpha+2-i_1-i_2}}\quad \forall y_h\in \Supp (1-\chi_n),
$$
we infer that
$$
|I_2|\leq C \|\zeta\|_{L^1} n^{-\alpha-2}\leq C \|\zeta\|_{L^1} |x_h|^{-\alpha-2}.
$$
Gathering all the terms, we obtain the inequality announced in the Lemma.
To conclude the proof, we still have to consider the case $\beta = -2$: in such a case, $|\xi|^\beta$ corresponds to inverting the Laplacian over $\R^2$. Hence, the kernel $|x_h - y_h|^{-\beta-2}$ has to be replaced by $\frac{1}{2\pi} \ln(|x_h - y_h|)$. This does not modify the previous reasoning. 
\end{proof}

\subsection*{Proof of Lemma \ref{lemmadecay2}}

The proof is somewhat simpler than the one of Lemma \ref{lemmadecay}. As indicated in the Remark following Lemma \ref{lemmadecay2}, notice that for $n>1$, for all $\xi\in \R^2, z>0$,
$$
\left|(1+|\xi|^2)^{-n} (1-\chi(\xi)) P(\xi) e^{-\lambda_j(\xi) z}\right|\leq \|P\|_{L^\infty(B_r^c)}\frac{e^{-\delta z}}{(1+|\xi|^2)^n},
$$
and the right-hand side of the above inequality is in $L^1(\R^2)$ for all $z$. As a consequence, for  $j=1..3$, $n>1$, the kernel
$$
K_{n,j}(x_h, z)\defin \int_{\R^2} e^{i x_h\cdot \xi} (1+|\xi|^2)^{-n}(1-\chi)(\xi) P(\xi) \exp(-\lambda_j(\xi)z)\:d\xi
$$
is well defined and satisfies
$$
\| K_{n,j}(\cdot, z)\|_{L^\infty(\R^2)}\leq C_n \|P\|_{L^\infty(B_r^c)} e^{-\delta z}.
$$
Furthermore, if $a_1, a_2\in \N$ with $a_1+a_2\leq 3$,
$$
x_1^{a_1} x_2^{a_2} K_{n,j}(x_h,z)= \int_{\R^2} e^{i x_h\cdot \xi} D_1^{a_1}D_2^{a_2}\left((1+|\xi|^2)^{-n}(1-\chi)(\xi) P(\xi) \exp(-\lambda_j(\xi)z)\right)\:d\xi.
$$
Hence, up to taking a larger $n$ and a smaller $\delta$: 
$$
|K_{n,j}(x_h,z)|\leq C_n \|P\|_{W^{3,\infty}(B_r^c)} e^{-\delta z}(1+|x_h|)^{-3},
$$
and in particular, $K_{n,j}\in L^\infty_z(L^2_{x_h})$. Thus for any $f\in L_{uloc}^2$,
$$
\|(1+ |D|^2)^{-n} (1-\chi(D)) P(D) \exp(-\lambda_j(D)z) f\|_{L^\infty}=\|K_{n,j}\ast f\|_{L^\infty}\leq C e^{-\delta z}\|f\|_{L^2_{uloc}}.
$$
Taking $f=(1+|D|^2)^n u_0=(1-\Delta_h)^n u_0$ for some $u_0\in H^{2n}_{uloc}$, we obtain the result announced in  Lemma \ref{lemmadecay2}.

\section*{Appendix B: Estimates on a few integrals}

\begin{lemma} There exists a positive constant $C$ such that for all $z\geq 0$,
$$
\ba
\int_0^\infty\frac{1}{(1+|z-z'|)^{2/3}(1+z')^{2/3}} dz'\leq \frac{C}{(1+z)^{1/3}},\\
\int_0^\infty\frac{1}{(1+|z-z'|)(1+z')^{2/3}} dz'\leq \frac{C\ln (2+z)}{(1+z)^{2/3}},
\ea
$$
and for all $\gamma, \delta >0$ such that $\delta<1$ and $\gamma + \delta>1$, there exists a constant $C_{\gamma, \delta}$ such that
$$
\int_0^\infty \frac{1}{(1+z+z')^\gamma}\frac{1}{(1+z')^\delta}\:dz'\leq \frac{C_{\gamma, \delta} }{(1+z)^{\gamma+ \delta}}\quad \forall z\geq 0.
$$
\label{lem:integrales}

\end{lemma}
\begin{proof}The first two inequalities are obvious if $z$ is small (say, $z\leq 1/2$), simply by writing
$$
\frac{1}{1+ |z-z'|}\leq\frac{C}{1+z'}.
$$
Hence we focus on $z'\geq 1/2$. In that case, changing variables in the first integral, we have
\begin{eqnarray*}
\int_0^\infty\frac{1}{(1+|z-z'|)^{2/3} (1+z')^{2/3}} dz'&=&\frac{1}{z^{1/3}}\int_0^\infty\frac{1}{\left({z}^{-1} + |1-t|\right)^{2/3}} \frac{1}{\left({z}^{-1} + t\right)^{2/3}} dt\\
&\leq&\frac{1}{z^{1/3}}\int_0^\infty\frac{1}{ |1-t|^{2/3}} \frac{1}{ t^{2/3}} dt,
\end{eqnarray*}
which proves the first inequality. The second one is treated in a similar fashion:
\begin{eqnarray*}
\int_0^\infty\frac{1}{1+|z-z'|}\frac{1}{(1+z)^{2/3}} dz'&=&z^{-1}\int_0^\infty\frac{1}{{z}^{-1} + |1-t|} \frac{1}{\left({z}^{-1} + t\right)^{2/3}} dt\\
&\leq&z^{-1}\int_0^\infty\frac{1}{{z}^{-1} + |1-t|} \frac{1}{ t^{2/3}} dt .
\end{eqnarray*}
It is easily checked that $\int_{1/2}^{3/2} \frac{1}{{z}^{-1} + |1-t|} dt\leq C \ln(2+z) $. The second estimate follows. The last estimate is proved by similar arguments and is left to the reader.

\end{proof}

\bibliography{ekmanNL}

\def\cprime{$'$} \def\cprime{$'$} \def\cprime{$'$}
\providecommand{\bysame}{\leavevmode\hbox to3em{\hrulefill}\thinspace}
\providecommand{\MR}{\relax\ifhmode\unskip\space\fi MR }
\providecommand{\MRhref}[2]{%
  \href{http://www.ams.org/mathscinet-getitem?mr=#1}{#2}
}
\providecommand{\href}[2]{#2}
\begin{thebibliography}{10}

\bibitem{Achdou:1998a}
Y.~Achdou, P.~Le~Tallec, F.~Valentin, and O.~Pironneau, \emph{Constructing wall
  laws with domain decomposition or asymptotic expansion techniques}, Comput.
  Methods Appl. Mech. Engrg. \textbf{151} (1998), no.~1-2, 215--232, Symposium
  on Advances in Computational Mechanics, Vol.\ 3 (Austin, TX, 1997).
  \MR{MR1625432 (99h:76027)}

\bibitem{Achdou:1995}
Y.~Achdou, B.~Mohammadi, O.~Pironneau, and F.~Valentin, \emph{Domain
  decomposition \& wall laws}, Recent developments in domain decomposition
  methods and flow problems (Kyoto, 1996; Anacapri, 1996), GAKUTO Internat.
  Ser. Math. Sci. Appl., vol.~11, Gakk\=otosho, Tokyo, 1998, pp.~1--14.
  \MR{MR1661558 (99h:76070)}

\bibitem{Achdou:1998}
Yves Achdou, O.~Pironneau, and F.~Valentin, \emph{Effective boundary conditions
  for laminar flows over periodic rough boundaries}, J. Comput. Phys.
  \textbf{147} (1998), no.~1, 187--218. \MR{MR1657773 (99j:76086)}

\bibitem{ABZ}
Thomas Alazard, Nicolas Burq, and Claude Zuily, \emph{Cauchy theory for the
  gravity water waves system with non localized initial data}, Ann. Inst. H.
  Poincaré Anal. Non Lin\'eaire (2014), to appear.

\bibitem{AmBrLe}
Youcef Amirat, Didier Bresch, J{\'e}r{\^o}me Lemoine, and Jacques Simon,
  \emph{Effect of rugosity on a flow governed by stationary {N}avier-{S}tokes
  equations}, Quart. Appl. Math. \textbf{59} (2001), no.~4, 769--785.
  \MR{1866556 (2002g:76036)}

\bibitem{BaGe}
Arnaud Basson and David G{\'e}rard-Varet, \emph{Wall laws for fluid flows at a
  boundary with random roughness}, Comm. Pure Appl. Math. \textbf{61} (2008),
  no.~7, 941--987. \MR{2410410 (2009h:76055)}

\bibitem{Bocquet_Barrat}
L.~Bocquet and J-L. Barrat, \emph{Flow boundary conditions : from nano- to
  micro- scales}, Soft Matter \textbf{3} (2003), no.~3, 180--183.

\bibitem{Bonnivard}
Matthieu Bonnivard and Dorin Bucur, \emph{The uniform rugosity effect}, J.
  Math. Fluid Mech. \textbf{14} (2012), no.~2, 201--215. \MR{2925104}

\bibitem{BoyerFabrie}
Franck Boyer and Pierre Fabrie, \emph{Mathematical tools for the study of the
  incompressible {N}avier-{S}tokes equations and related models}, Applied
  Mathematical Sciences, vol. 183, Springer, New York, 2013.

\bibitem{BrMi}
Didier Bresch and Vuk Milisic, \emph{High order multi-scale wall-laws, {P}art
  {I}: the periodic case}, Quart. Appl. Math. \textbf{68} (2010), no.~2,
  229--253. \MR{2663000 (2011k:76029)}

\bibitem{Bucur_Feireisl}
Dorin Bucur, Eduard Feireisl, {\v{S}}{\'a}rka Ne{\v{c}}asov{\'a}, and Joerg
  Wolf, \emph{On the asymptotic limit of the {N}avier-{S}tokes system on
  domains with rough boundaries}, J. Differential Equations \textbf{244}
  (2008), no.~11, 2890--2908. \MR{2418180 (2009h:35311)}

\bibitem{Casado}
Juan Casado-D{\'{\i}}az, Enrique Fern{\'a}ndez-Cara, and Jacques Simon,
  \emph{Why viscous fluids adhere to rugose walls: a mathematical explanation},
  J. Differential Equations \textbf{189} (2003), no.~2, 526--537. \MR{1964478
  (2003k:76039)}

\bibitem{DaGe}
Anne-Laure Dalibard and David G{\'e}rard-Varet, \emph{Effective boundary
  condition at a rough surface starting from a slip condition}, J. Differential
  Equations \textbf{251} (2011), no.~12, 3450--3487. \MR{2837691 (2012k:35031)}

\bibitem{DalibardPrange}
Anne-Laure Dalibard and Christophe Prange, \emph{Well-posedness of the
  {S}tokes-{C}oriolis system in the half-space over a rough surface}, Anal. PDE
  \textbf{7} (2014), no.~6, 1253--1315.

\bibitem{DroniouImbert}
J{\'e}r{\^o}me Droniou and Cyril Imbert, \emph{Fractal first-order partial
  differential equations}, Arch. Ration. Mech. Anal. \textbf{182} (2006),
  no.~2, 299--331.

\bibitem{Galdi}
G.~P. Galdi, \emph{An introduction to the mathematical theory of the
  {N}avier-{S}tokes equations}, second ed., Springer Monographs in Mathematics,
  Springer, New York, 2011, Steady-state problems.

\bibitem{DGV_Dormy}
D.~G{\'e}rard-Varet and E.~Dormy, \emph{Ekman layers near wavy boundaries}, J.
  Fluid Mech. \textbf{565} (2006), 115--134. \MR{2263528 (2007i:76106)}

\bibitem{Gerard-Varet:2003b}
David G{\'e}rard-Varet, \emph{Highly rotating fluids in rough domains}, J.
  Math. Pures Appl. (9) \textbf{82} (2003), no.~11, 1453--1498. \MR{MR2020807
  (2004j:76162)}

\bibitem{DGV_CMP}
\bysame, \emph{The {N}avier wall law at a boundary with random roughness},
  Comm. Math. Phys. \textbf{286} (2009), no.~1, 81--110. \MR{2470924
  (2010b:35340)}

\bibitem{DGVNMnoslip}
David G{\'e}rard-Varet and Nader Masmoudi, \emph{Relevance of the slip
  condition for fluid flows near an irregular boundary}, Comm. Math. Phys.
  \textbf{295} (2010), no.~1, 99--137.

\bibitem{Jager:2001}
Willi J{\"a}ger and Andro Mikeli{\'c}, \emph{On the roughness-induced effective
  boundary conditions for an incompressible viscous flow}, J. Differential
  Equations \textbf{170} (2001), no.~1, 96--122. \MR{MR1813101 (2002b:76049)}

\bibitem{Jager:2003}
\bysame, \emph{Couette flows over a rough boundary and drag reduction}, Comm.
  Math. Phys. \textbf{232} (2003), no.~3, 429--455. \MR{MR1952473
  (2003j:76025)}

\bibitem{LS}
O.~A. Lady{\v{z}}enskaja and V.~A. Solonnikov, \emph{Determination of solutions
  of boundary value problems for stationary {S}tokes and {N}avier-{S}tokes
  equations having an unbounded {D}irichlet integral}, Zap. Nauchn. Sem.
  Leningrad. Otdel. Mat. Inst. Steklov. (LOMI) \textbf{96} (1980), 117--160,
  308, Boundary value problems of mathematical physics and related questions in
  the theory of functions, 12.

\bibitem{Lauga}
E.~Lauga, M.P. Brenner, and H.A. Stone, \emph{Microfluidics: The no-slip
  boundary condition}, Handbook of Experimental Fluid Dynamics,C. Tropea, A.
  Yarin, J. F. Foss (Eds.), Springer, 2007, p.~123601.

\bibitem{MikNec}
Andro Mikeli{\'c}, {\v{S}}{\'a}rka Ne{\v{c}}asov{\'a}, and Maria Neuss-Radu,
  \emph{Effective slip law for general viscous flows over an oscillating
  surface}, Math. Methods Appl. Sci. \textbf{36} (2013), no.~15, 2086--2100.
  \MR{3108828}

\bibitem{Neuss}
N.~Neuss, M.~Neuss-Radu, and A.~Mikeli{\'c}, \emph{Effective laws for the
  {P}oisson equation on domains with curved oscillating boundaries}, Appl.
  Anal. \textbf{85} (2006), no.~5, 479--502. \MR{2213071 (2006k:35012)}

\bibitem{Ped}
J.~Pedloski, \emph{Geophysical fluid dynamics}, Springer-Verlag, New York,
  1992.

\bibitem{Prange}
Christophe Prange, \emph{Asymptotic analysis of boundary layer correctors in
  periodic homogenization}, SIAM J. Math. Anal. \textbf{45} (2013), no.~1,
  345--387. \MR{3032981}

\bibitem{strichartz}
Robert~S. Strichartz, \emph{Multipliers on fractional {S}obolev spaces}, J.
  Math. Mech. \textbf{16} (1967), 1031--1060. \MR{0215084 (35 \#5927)}

\bibitem{Vino}
O.~Vinogradova and G.~Yakubov, \emph{Surface roughness and hydrodynamic
  boundary conditions}, Phys. Rev. E (2006), 479--487.

\bibitem{Ybert}
C.~Ybert, C.~Barentin, C.~Cottin-Bizonne, P.~Joseph, and Bocquet L.,
  \emph{Achieving large slip with superhydrophobic surfaces: Scaling laws for
  generic geometries}, Physics of fluids \textbf{19} (2007), 123601.

\end{thebibliography}

\end{document}